\documentclass[11pt,leqno]{amsart}
\usepackage{amsmath}
\usepackage{amsfonts}
\usepackage{amssymb}
\usepackage{graphicx}
\usepackage{color}
\newtheorem{theorem}{Theorem}[section]\newtheorem{thm}[theorem]{Theorem}
\newtheorem*{theorem*}{Theorem}
\newtheorem{lemma}{Lemma}[section]
\newtheorem{corollary}[theorem]{Corollary}

\newtheorem{prop}{Proposition}[section]
\newtheorem{definition}[theorem]{Definition}

\newtheorem{remark}[theorem]{Remark}

\def \b {\beta}

\def\k{\kappa}
\def\n{\nabla}
\def\Ric{\text{Ric}}

\def\a{\alpha}
\def\l{\lambda}
\def\ol{\overline}

\def\e{\epsilon}
\def\p{\partial}

\def\R{\mathbb{R}}

\def\vp{\varphi}

\def\L{{\mathcal L}}

\def\Ric{\operatorname{Ric}}

\def\tr{\operatorname{tr}}

\newcommand{\eps}{{\varepsilon}}

\numberwithin{equation}{section}
\def\F{\mathcal{F}}

\begin{document}

\title[On a class of quasilinear operators]{On a class of quasilinear operators on smooth metric measure spaces}

\author{Xiaolong Li}
\address{(to start) Department of Mathematics and Statistics, McMaster University, Hamilton, Ontario, L8S 4K1, Canada}
\email{lxlthu@gmail}

\author{Yucheng Tu}
\address{Department of Mathematics, University of California, San Diego, La Jolla, CA 92093, USA}
\email{y7tu@ucsd.edu}

\author{Kui Wang}\thanks{The research of the third author is supported by NSFC No.11601359} 
\address{School of Mathematical Sciences, Soochow University, Suzhou, 215006, China}
\email{kuiwang@suda.edu.cn}




\maketitle

\begin{abstract}  
We derive sharp estimates on the modulus of continuity for solutions of a large class of quasilinear isotropic parabolic equations on smooth metric measure spaces (with Dirichlet  or Neumann boundary condition in case the boundary is non-empty). 
We also derive optimal lower bounds for the first Dirichlet eigenvalue of a class of homogeneous quasilinear operators, which include non-variational operators. 
The main feature is that this class of operators have corresponding one-dimensional operators, which allow sharp comparisons with solutions of one-dimensional equations. 
\end{abstract}

\tableofcontents

\section{Introduction}
A smooth metric measure space is a triple $(M^n, g, e^{-f} d\mu_g)$, where $(M^n,g)$ is a complete $n$-dimensional Riemannian manifold, $f$ is a smooth real function on $M$, and $d\mu_g$ is the Riemannian volume density on $M$. The study of partial differential equations on smooth metric measure spaces has received considerable attention in the last few decades, and numerous results on Riemannian manifolds with Ricci curvature bounded from below were extended to smooth metric measure spaces (many to even non-smooth metric measure spaces) with $N$-Bakry-\'Emery Ricci curvature bounded from below. Recall that the $N$-Barky-\'Emery Ricci curvature $\Ric^N_f$ of a smooth metric measure space, which is a natural generalization of the classical Ricci curvature for Riemannian manifolds, is defined for $N \in [n, \infty]$ by
\begin{equation*}
    \Ric^N_f =\begin{cases} \Ric, & \text{ if } N=n \text{ and } f=\text{constant}, \\
    -\infty, & \text{ if } N=n \text{ and } f\neq \text{constant},\\ 
    \Ric +\nabla^2 f -\frac{\nabla f \otimes \nabla f}{N-n}, & \text{ if } N \in (n, \infty), \\
    \Ric+\nabla^2 f, & \text{ if } N=\infty.
     \end{cases} \\
\end{equation*}
Here $\Ric$ and $\nabla^2 f$ denote the Ricci curvature of $M$ and the Hessian of $f$, respectively. 
We also write $\Ric_f := \Ric^{\infty}_f = \Ric + \nabla^2 f$ for simplicity. 
Important examples of smooth metric measure spaces with $\Ric^N_f$ bounded from below include: Riemannian manifolds with Ricci curvature bounded from below (corresponds to $N=n$ and $f=\text{constant}$), Bakry-\'Emery manifolds (corresponds to $N=\infty$), gradient Ricci solitons (i.e. $\Ric_f = \k g$ for $\k \in \R$), and quasi-Einstein metrics (i.e. $\Ric+\n^2 f -\frac{\nabla f \otimes \nabla f}{N-n} =\k g$ for $\k \in \R$). 

On a smooth metric measure space $(M^n ,g, e^{-f} d\mu_g)$, we are interested in the following quasi-linear isotropic operators $Q: TM \setminus\{0\} \times \text{Sym}^2T^*M \to \R$
\begin{equation}\label{Q def}
Q[p, X] := \tr \left( \left[\a(|p|)\frac{p \otimes p}{|p|^2}+\b(|p|)\left(I_n-\frac{p \otimes p}{|p|^2}  \right) \right] X \right)-\b(|p|)\langle p, \n f \rangle, 
\end{equation}
where $\a$ and $\b$ are nonnegative continuous functions, $I_n$ is the $n\times n$ identity matrix, and $TM$, $T^*M$ and $\text{Sym}^2T^*M$ denote the tangent bundle, the cotangent bundle and the set of symmetric two tensors on $M$, respectively. 
Throughout the paper, we assume that $Q: TM \setminus\{0\} \times \text{Sym}^2T^*M \to \R$ is a continuous function (allowed to be singular at $p=0$) and $Q$ is degenerate elliptic in the sense that $Q(p,X) \geq Q(p,Y)$ for all $p \in TM \setminus\{0\}$ and $X \geq Y \in \text{Sym}^2T^*M$. Since $Q[u]$ is not necessarily of divergent form, it is necessary to the terminology of viscosity solutions in \cite{CIL92} (see also Section 2) throughout the paper. 
Note that the family of operators $Q[u]:=Q[\n u, \n^2 u]$ in \eqref{Q def} includes many greatly studied elliptic operators for suitable choices of $\a, \b$ and $f$. For instance, $Q[u]$ covers:
\begin{enumerate}
    \item the Laplacian \\
    $\Delta u:=\text{div}(\n u)$ (with $\a =\b =1$ and $f=\text{constant}$);
    \item the $f$-Laplacian \\
    $\Delta_f u :=\Delta u -\langle \n u, \n f \rangle$ (with $\a =\b =1$);
    \item the $p$-Laplacian \\
    $\Delta_p u :=\text{div}(|\n u|^{p-2} \n u)$ with $1<p<\infty$ (with $\a =(p-1)|\n u|^{p-2}$, $\b =|\n u|^{p-2}$ and $f=\text{constant}$);
    \item the weighted $p$-Laplacian \\
    $\Delta_{p,f}u :=\Delta_p u -|\n u|^{p-2} \langle \n u, \n f \rangle$ (with $\a =(p-1)|\n u|^{p-2}$, $\b =|\n u|^{p-2}$);
    \item the mean curvature operator \\
    $Hu :=\sqrt{1+|\n u|^2}\text{div}\left( \frac{\n u}{\sqrt{1+|\n u|^2}}\right)$ (with  $\a =\frac{1}{1+|
    \n u|^2}$, $\b=1$ and $f =\text{constant}$ ),
\end{enumerate}
as well as some non-divergent or degenerate elliptic operators such as 
\begin{enumerate}
    \item[(6)] the normalized or game-theoretic $p$-Laplacian \\
    $\Delta^N_p u:=\frac{1}{p}|\n u |^{2-p} \text{div}(|\n u|^{p-2} \n u)$ for $1<p<\infty$ (with $\a=(p-1)/p$ and $\b=1/p$ and $f=\text{constant}$);
    \item[(7)] the level set mean curvature operator or the $1$-Laplacian \\
    $\Delta_1 u:= \Delta u -|\n u|^{-2} \n^2 u(\n u, \n u)$ (with $\a =0$ and $\b=1$, $f=\text{constant}$);
    \item[(8)] the $\infty$-Laplacian \\
    $\Delta_\infty u:= |\n u|^{-2} \n^2 u(\n u, \n u)$ (with $\a=1$, $\b=0$ and $f=\text{constant}$) 
\end{enumerate}
Indeed, the second order part of $Q[u]$ can be written as
a combination of the $1$-Laplacian 
and the $\infty$-Laplacian, 
\begin{equation*}
    Q[u]=\a(|\n u|) \Delta_{\infty} u + \b(|\n u|) \Delta_{1} u - \b(|\n u|) \langle \n u, \nabla f \rangle.
\end{equation*}
The $1$-Laplacian $\Delta_1$ appears in the level set formulation of the mean curvature flow (see \cite{CGG} and \cite{ES91}) and have been investigated extensively ever since, while the $\infty$-Laplacian $\Delta_\infty$ plays an important role in the description of tug-of-war games (see for instance \cite{Evans07} and \cite{PSSW09}).
The normalized $p$-Laplacian $\Delta^N_p$ has recently received considerable attention for its important applications in image-processing \cite{Kawohl08} and in the description of tug-of-war games with noise \cite{PS08}. 

The class of operators $Q[u]$ in \eqref{Q def} were considered in various recent works including \cite{Andrewssurvey15}\cite{AC09}\cite{AC13}\cite{AN12}\cite{AX19}\cite{Li16}\cite{LW17}\cite{LW19eigenvalue}\cite{LW19eigenvalue2}. The key feature is that this class of operators has corresponding one-dimensional operators, which are obtained by assuming that the solution depends only on one of the variables, as well as taking into account the effect of geometric data such as curvature and dimension.  
This observation, together with the idea of comparing with solutions of one-dimensional equations, has led to numerous important results in the last decade. These include sharp gradient estimates via the the modulus of continuity estimates for quasi-linear equations on Euclidean domains \cite{AC09}, sharp modulus of continuity estimates for parabolic equations on Riemannian manifolds (assuming either $\p M=\emptyset$ or $\p M$ is convex and the Neumann boundary condition is imposed) and a simple proof of the optimal lower bound for the first nonzero eigenvalue of the Laplacian \cite{AC13}, proof of the Fundamental Gap Conjecture for convex domains in the Euclidean space \cite{AC11} (see also \cite{Ni13} for an elliptic proof and \cite{DSW18}\cite{HW17}\cite{SWW19} for the fundamental gaps of convex domains in the sphere). 
The above-mentioned results are proved by using the two-points maximum principle together with comparisons with one-dimensional models.  
We refer the reader to the wonderful survey \cite{Andrewssurvey15} for more discussions on modulus of continuity estimates and its applications as well as more applications of the two-points maximum principle in geometric heat equations. 

The purpose of the present paper is to investigate the quasi-linear operators defined in \eqref{Q def} with the Dirichlet boundary condition on smooth metric measure spaces, as well as to extend previous results for Riemannian manifolds in \cite{AC13}\cite{AX19}\cite{LW17}\cite{LW19eigenvalue} to the more general smooth metric measure spaces setting. 

Recall that a nonnegative function $\vp$ is called a modulus of continuity for a function $u: M \to \R$ if
\begin{equation*}
    u(y)-u(x) \leq 2 \vp \left(\frac{d(x,y)}{2} \right)
\end{equation*}
for all $x, y \in M$, where $d$ is the distance induced by the Riemannian metric. 
Our first result establishes sharp estimates on the modulus of continuity of solutions to the parabolic equation 
\begin{equation}\label{parabolic pde}
    u_t =Q[u]
\end{equation}
with Dirichlet boundary condition. 

\begin{thm}\label{Thm MC}
Let $(M^n,g, e^{-f} d\mu_g)$ be a compact smooth metric measure space with smooth nonempty boundary $\p M$ and diameter $D$. 
Let $u:M\times [0,T) \to \R$ be a viscosity solution of \eqref{parabolic pde}
with Dirichlet boundary condition $u(x,t)=0$ for $x\in \p M$ and $t\in [0,T)$. Let $\vp:[0, D/2] \times \R_+ \to \R_+$ be a smooth function satisfying (i) $\vp_0(s):=\vp(s, 0)$ is a modulus of continuity of $u(x, 0)$ and $|u(x,0)|\le \vp_0(d(x,\partial M))$; (ii) $\vp'\geq 0$ and $\vp''\le 0$ on $[0, D/2] \times \R_+$, $\vp(0,t)=0$ for $t\ge 0$. 
\begin{enumerate}
    \item[(1)] Suppose that for $N \in [n, \infty)$, we have $\Ric^N_f \geq (N-1)\kappa$ and $H_f \geq (N-1) \Lambda$ for $\k, \Lambda \in \R$, and $\vp$ satisfies $\vp_t \geq \a(\vp')\vp'' -(N-1) T_{\kappa, \Lambda} \b(\vp')\vp'$. Then $\vp(s,t)$ is a modulus of continuity for $u(x,t)$ for each $t \in [0,T)$;
    \item[(2)] Suppose that $\Ric_f \geq 0$, $H_f \geq 0$ and $\vp$ satisfies $\vp_t \geq \a(\vp')\vp''$. Then $\vp(s,t)$ is a modulus of continuity for $u(x,t)$ for each $t \in [0,T)$.
\end{enumerate}
\end{thm}

Here $H_f$ denotes the $f$-mean curvature of $\p M$ defined by
\begin{equation*}
    H_f(x) =H(x) -\langle \n f(x),  \nu(x) \rangle,
\end{equation*}
where $\nu(x)$ is the outward unit normal vector field at $x\in \p M$ and $H(x)$ denotes the mean curvature at $x \in \p M$, and the function $T_{\kappa,\Lambda}$ is defined for $\kappa, \Lambda \in \R$ by 
$$T_{\kappa,\Lambda} (t):=- \frac{C'_{\kappa, \Lambda}(t)}{C_{\kappa, \Lambda}(t)}, $$ 
where $C_{\kappa, \Lambda}(t)$ is unique solution of the initial value problem
\begin{equation}\label{C def}
    \begin{cases} 
    \phi''+\kappa \phi =0, \\
    \phi(0)=1,     \phi'(0) =-\Lambda.
    \end{cases}
\end{equation}

Theorem \ref{Thm MC} complements the work of Andrews and Clutterbuck \cite{AC13} for Riemannian manifold (closed or with a convex boundary and Neumann boundary condition). 
It seems for the Neumann case, one can only deal with convex boundaries as in \cite{AC13}, but for the Dirichlet case, we are able to handle any lower bound of the $f$-mean curvature of the boundary.  Theorem \ref{Thm MC} is sharp. In fact, by exact same process as in \cite[Section 5]{AC13}, one can construct solutions of    \eqref{parabolic pde} on 
$(\k, \Lambda)$-equational model space as in \cite[Theorem 1.6 and 1.7]{Sakurai19}, satisfying the conditions of Theorem \ref{Thm MC} and satisfying the conclusion with equality holds.

The proof of Theorem \ref{Thm MC} relies on the two comparison theorems, one for $d(x,\p M)$ (see Theorem \ref{Thm comparison distance to boundary}), the other one for $d(x,y)$ as a function on $M\times M$ (see Theorem \ref{Thm comparison distance}). Both comparison theorems are sharp, more general than the ones in the literature and of independent interests.

The modulus of continuity estimates in 
\cite{AC13} led to an easy proof of the optimal lower bound on the first nonzero (closed or Neumann) eigenvalue of the Laplacian or the $f$-Laplacian in terms of dimension, diameter, and the lower bound of Ricci curvature. 
The sharp lower bound was previously proved by Zhong and Yang for the nonnegative Ricci case case, by Kr\"oger \cite{Kroger92} (see also Bakry-Qian \cite{BQ00} for an explicit statement) for general Ricci lower bound using gradient estimates method, 
and independently by Chen and Wang \cite{CW94,CW95} using stochastic method. 
The gradient estimate method, dating back to the work of Li \cite{Li79} and Li and Yau \cite{LY80}, was used to prove sharp lower bound for the first nonzero eigenvalue of the $p$-Laplacian 
in \cite{Valtorta12} and \cite{NV14}, and for the weighted $p$-Laplacian in \cite{LW19eigenvalue, LW19eigenvalue2}. On the other hand, the modulus of continuity approach seems only work for $1<p\leq 2$, see \cite[Section 8]{Andrewssurvey15} and \cite[Section 2]{LW19eigenvalue}. 


It is natural to ask whether the modulus of continuity estimates in Theorem \ref{Thm MC}
lead to lower bound for the first Dirichlet eigenvalue of the Laplacian or more generally the $p$-Laplacian.  
It turns out that we can establish optimal lower bound for the first Dirichlet eigenvalue of a large class of quasi-linear operators, but not as a consequence of Theorem \ref{Thm MC}. 
Indeed, this can be achieved this by proving similar estimates as in Theorem \ref{Thm MC} but requiring one of the two variables to be contained in $\p M$. We focus on the statement here and elaborate the difference between the closed/Neumann and Dirichlet case in Section 6. 
To define eigenvalues, it is necessary to assume that $Q[u]$ is homogeneous of degree $\gamma >0$ in the sense that
\begin{equation}\label{def homogeneous}
    Q[c u] =c^{\gamma} Q[u].
\end{equation}
The Dirichlet eigenvalue problem is then 
\begin{equation}\label{eigen problem}
\begin{cases}
 Q[u]  =-\l |u|^{\gamma -1 } u, & \text{ in } M, \\
 u=0,  & \text{ on } \p M.  
\end{cases}
\end{equation}
Examples of homogeneous variational operators include the Laplacian, the $f$-Laplacian, the $p$-Laplacian and the weighted $p$-Laplacian, while examples of homogeneous non-variational operators include the normalized $p$-Laplacian $\Delta^N_p$ for $1<p<\infty$ and the operator $|Du|^{\gamma} \Delta^N_p$ for $\gamma > -1$ and $1<p<\infty$ (homogeneous of degree $\gamma+1$).
For operators that are not variational, a new definition for its first Dirichlet eigenvalue (also called the principle eigenvalue) is needed. 
Following Berestycki, Nirenberg, and Varadhan \cite{BNV94}, in the papers \cite{BD07, BD06} (where they actually deal with a wider class of operators), Birindelli and Demengel introduced the first Dirichlet eigenvalue $\bar{\l}(Q)$ of the operator $Q[u]$  defined as
\begin{equation}\label{def principal eigenvalue}
\bar{\l}(Q) =\sup \{\l \in \R : \text{there exists a positive viscosity supersolution $u$  of } \eqref{eigen problem} \}.
\end{equation}
Calling it a first Dirichlet eigenvalue can be justified: they proved 
that there exists a positive eigenfunction vanishing on the boundary associated with $\bar{\l}$, via Perron's method for viscosity solutions. In other words, for $\l=\bar{\l}(Q)$, the eigenvalue problem \eqref{eigen problem} admits a positive viscosity solution. The simplicity of $\bar{\l}(Q)$ has been proved very recently for the normalized $p$-Laplacian in \cite{CFK20}, but is not known for general operators. 

We establish the following optimal lower bound for $\bar{\l}(Q)$ in terms of geometric data of the underlying smooth metric measure space. 
\begin{thm}\label{thm Dirichlet eigenvalue}
Let $(M^n, g, e^{-f}d\mu_g)$ be a compact smooth metric measure space with smooth nonempty boundary $\p M$. 
Let $Q[u]$ be defined in \eqref{Q def} and assume further that $Q[u]$ is homogeneous of degree $\gamma >0$ in the sense of \eqref{def homogeneous}.  
Let $\bar{\l}(Q)$ be the first Dirichlet eigenvalue of $Q[u]$ defines as in \eqref{def principal eigenvalue}. 
\begin{enumerate}
    \item[(i)] Suppose $\Ric^N_f \geq (N-1)\kappa$ and $H_f \geq \Lambda$ for $N \in [n,\infty)$ and $\k, \Lambda \in \R$. Then we have 
\begin{equation*}
    \bar{\l}(Q) \geq \l_1
\end{equation*}
where $\l_1$ is the first eigenvalue of the one-dimensional problem
\begin{equation}\label{1D eq N finite}
     \begin{cases}
     \a (\vp') \vp'' -(N-1) T_{\k,\Lambda} \b(\vp') \vp' =-\l |\vp|^{\gamma-1}\vp, \\
     \vp(0)=0, \vp'(R)=0.
     \end{cases}
\end{equation}
\item[(ii)] Suppose $\Ric^N_f \geq 0$ and $H_f \geq 0$. Then we have 
\begin{equation*}
    \bar{\l}(Q) \geq \mu_1
\end{equation*}
where $\mu_1$ is the first eigenvalue of the one-dimensional problem on $[0,R]$
\begin{equation}\label{1D eq N infinite}
     \begin{cases}
     \a (\vp') \vp'' =-\l |\vp|^{\gamma-1}\vp,  \\
     \vp(0)=0, \vp'(R)=0.
     \end{cases}
\end{equation}
\end{enumerate}
\end{thm}

Theorem \ref{thm Dirichlet eigenvalue} covers the sharp lower bound of the first Dirichlet eigenvalue of the Laplacian proved by Li and Yau \cite{LY80} for $\k=\Lambda=0$ and  by Kasue \cite{Kasue84} for general $\k, \Lambda \in \R$ and of the $p$-Laplacian and weighted $p$-Laplacian by Sakurai \cite{Sakurai19}. 
Furthermore, the equality case is achieved if and only if  $(M^n, g, e^{-f}d\mu_g)$ is a $(\k, \Lambda)$-equational model space. 

We will indeed provide two proofs for Theorem \ref{thm Dirichlet eigenvalue} in Section 6, one uses the idea of modulus of continuity estimates by studying sharp decay rate (see Theorem \ref{Thm Decay Intro}) for solutions of the parabolic equation \eqref{parabolic pde} as in the closed or Neumann case, while the other one uses the new definition \eqref{def principal eigenvalue} together with Theorem \ref{Thm comparison distance to boundary}, which easily provides a positive viscosity super-solution to the eigenvalue problem \eqref{eigen problem}.  

Finally, we extend the results in \cite{AC09}\cite{AC13}\cite{AX19}\cite{LW17}\cite{LW19eigenvalue} on Riemannian manifolds to the more general setting of smooth metric measure spaces.
These include sharp modulus of continuity for solutions of \eqref{parabolic pde} with empty boundary or convex boundary and Neumann boundary condition (see Theorem \ref{thm Neumann}), 
sharp height-dependent gradient bounds for parabolic equations (see Theorem \ref{thmh}) and for elliptic equations (see Theorem \ref{Thm1.3} and \ref{Thm3.1}).


The paper is organized as follows. 
In Section 2, we recall the definition of viscosity solutions and the parabolic maximum principle for semi-continuous functions. 
In Section 3 and 4, we prove comparison theorems for the second derivatives of $d(x,\p M)$ and $d(x,y)$, respectively. 
In Section 5, we extend the results in \cite{AC13} on Riemannian manifolds to smooth metric measure spaces. 
The proof of Theorem \ref{Thm MC} will be given in Section 6. 
Then two proofs of Theorem \ref{thm Dirichlet eigenvalue} are provided in Section 7. 
In Sections 8 and 9, we prove sharp height-dependent gradient bounds for parabolic and elliptic equations on smooth metric measure spaces, respectively.

\section{Preliminaries on Viscosity Solutions}

It is necessary to work with viscosity solutions as the operator $Q[u]$ is not necessarily of divergence form. We refer the reader to
\cite{CIL92} for the general theory of viscosity solutions of non-singular operators on domains in the Euclidean space, to \cite{Giga06} for necessary adaptions for singular operators and to \cite{I} for adaptions for equations on Riemannian manifolds. 

For the convenience of the reader, we provide in this section the definition of viscosity solutions and the parabolic maximum principle for semi-continuous functions on manifolds. 
\subsection{Definition of Viscosity Solutions}

Let $(M^n,g)$ be a Riemannian manifold. 
We write $\text{USC}(M\times (0,T))$ for the set of all upper-semicontinuous functions from $M\times (0,T)$ to $\R$. Likewise, $\text{LSC}(M\times (0,T))$ contains all lower-semicontinuous functions from $M\times (0,T)$ to $\R$. For upper- and lower-semicontinuous functions, one can introduce the notion of super- and sub-jets respectively (see for instance \cite[Section 8]{CIL92}).

We first introduce the notion of parabolic semijets on manifolds. We write $z=(x,t)$ and $z_0=(x_0, t_0)$.
\begin{definition}
For a function $u\in \mbox{USC}(M\times (0,T))$,
we define the parabolic second order superjet of $u$ at a point $z_0\in M\times (0,T)$ by
\begin{align*}
\mathcal{P}^{2,+} u (z_0) &:=\{(\vp_t(z_0), \n \vp(z_0), \n^2\vp(z_0)) :
 \vp \in C^{2,1}(M\times (0,T)),  \\
  & \mbox{  such that  } u- \vp \mbox{  attains a local maximum at } z_0\}.
\end{align*}
For $u\in \mbox{LSC}(M\times (0,T))$, the parabolic second order subjet of $u$ at $z_0\in M\times (0,T)$ is defined by
$$\mathcal{P}^{2,-} u (z_0):=-\mathcal{P}^{2,+} (-u) (z_0).$$
\end{definition}
We also define the closures of $\mathcal{P}^{2,+} u (z_0)$ and $\mathcal{P}^{2,-} u (z_0)$ by
\begin{align*}
\overline{\mathcal{P}}^{2,+}u(z_0)
&=\{(\tau,p,X)\in \R \times T_{x_0}M \times Sym^2(T^*_{x_0}M) |
\mbox{  there is a sequence  } (z_j,\tau_j,p_j,X_j) \\
&\mbox{  such that  } (\tau_j,p_j ,X_j)\in \mathcal{P}^{2,+}u(z_j) \\
&\mbox{  and  } (z_j,u(z_j),\tau_j,p_j,X_j) \to (z_0,u(z_0),\tau, p ,X) \mbox{  as  } j\to \infty \}; \\
\overline{\mathcal{P}}^{2,-}u(z_0)&=-\overline{\mathcal{P}}^{2,+}(-u)(z_0).
\end{align*}

\begin{definition}(Semicontinuous Envelopes)
Let $X$ be a metric space and let $f$ be a function defined on a dense subset of $X$. We call the function $f^*: X \to \R$ defined by
\begin{equation*}
    f^{*}(x) := \inf \{g \in  \text{USC}(X) | g \geq f\}
\end{equation*}
the upper-semicontinuous envelope of $f$. Analogously, 
\begin{equation*}
    f_{*}(x) := \sup \{g \in \text{LSC}(X) | g \leq f\}
\end{equation*}
is the lower-semicontinuous envelope of $f$.
\end{definition}

Now we give the definition of a viscosity solution for the general equation
\begin{equation} \label{equ mfd}
u_t+F(x, t, u, \n u, \n^2 u)=0
\end{equation}
on $M$.
Assume $F: M \times[0,T] \times \R \times \left(T_{x_0}M\setminus\{0\}\right) \times \text{Sym}^2(T^*_{x_0}M) \to \R$ is continuous and proper, i.e., 
$$F(x,t,r,p,X) \leq F(x,t,s,p,Y) \mbox{  whenever  } p\neq 0, r\leq s, Y \leq X.$$

\begin{definition}\label{def viscosity}
(i) A function $u \in \mbox{USC}(M\times(0,T))$ is a viscosity subsolution of \eqref{equ mfd}
if for all $z \in M\times(0,T)$ and $(\tau, p, X) \in \mathcal{P}^{2,+}u(z)$,
\begin{align*}
\tau +F_*(z, u(z), p, X) \leq 0.
\end{align*}

(ii) A function $u \in \mbox{LSC}(M\times(0,T))$ is a viscosity supersolution of \eqref{equ mfd}
if for all $z \in M\times(0,T)$ and $(\tau, p, X) \in \mathcal{P}^{2,-}u(z)$,
\begin{align*}
\tau +F^*(z, u(z) , p, X) \geq 0.
\end{align*}

(iii) A viscosity solution of \eqref{equ mfd} is defined to be a continuous function that is both a
viscosity subsolution and a viscosity supersolution of \eqref{equ mfd}.
\end{definition}

\subsection{Maximum principle for viscosity solutions}
The main technical tool we use is the parabolic maximum principle for semicontinuous functions on manifolds, which is a restatement of \cite[Theorem 8.3]{CIL92}, for Riemannian manifolds. One can also find it in \cite[Section 2.2]{I} or \cite[Theorem 3.8]{AFS1}.
\begin{thm}\label{max prin}
Let $M_1^{N_1}, \cdots, M_k^{N_k}$ be Riemannian manifolds, and $\Omega_i \subset M_i$ open subsets.
Let $u_i \in USC((0,T)\times \Omega_i)$, and $\vp$ defined on $(0,T)\times \Omega_1 \times \cdots \times \Omega_k$ such that $\vp$ is continuously differentiable in $t$ and twice continuously differentiable in
$(x_1, \cdots x_k) \in \Omega_1 \times \cdots \times \Omega_k$.
Suppose that $\hat{t} \in (0,T), \hat{x}_i \in \Omega_i$ for $i=1, \cdots, k$ and the function
$$\omega(t, x_1, \cdots, x_k) :=u_1(t,x_1)+\cdots + u_k(t,x_k)-\vp(t,x_1, \cdots , x_k) $$
attains a maximum at $(\hat{t},\hat{x}_1, \cdots, \hat{x}_k)$ on $(0,T)\times \Omega_1 \times \cdots \times \Omega_k$.
Assume further that there is an $r >0$  such that for every $\eta >0$ there is a $C>0$ such that for $i=1, \cdots, k$
\begin{align*}
& b_i \leq C  \mbox{  whenever  } (b_i,q_i,X_i) \in \ol{\mathcal{P}}^{2,+}u_i(t,x_i) , \\
& d(x_i, \hat{x}_i)+|t-\hat{t}| \leq r \mbox{  and  } |u_i(t,x_i)|+|q_i| +\|X_i\| \leq \eta.
\end{align*}
Then for each $\lambda>0$, there are $X_i \in Sym^2(T^*_{\hat{x}_i} M_i)$ such that
\begin{align*}
& (b_i,\n_{x_i}\vp(\hat{t},\hat{x}_1, \cdots, \hat{x}_k),X_i)  \in \ol{\mathcal{P}}^{2,+}u_i(\hat{t},\hat{x}_i),\\
&  -\left(\frac 1 \lambda +\left\|S\right\| \right)I \leq
    \begin{pmatrix}
   X_1 & \cdots & 0 \\
   \vdots & \ddots & \vdots \\
   0 & \cdots & X_k
   \end{pmatrix}
   \leq S+\lambda S^2,  \\
& b_1 + \cdots + b_k =\vp_t(\hat{t},\hat{x}_1, \cdots, \hat{x}_k),
\end{align*}
where $S=\n^2\vp(\hat{t},\hat{x}_1, \cdots, \hat{x}_k).$
\end{thm}

\section{Comparison theorems for the second derivatives of $d(x, \p M)$}

In this section, we prove comparison theorems for the second derivatives of $d(x,\p M)$. Let $R$ denote the inradius of $M$ defined by 
$$R=\sup\{d(x, \p M): x\in M \}.$$

\begin{thm}\label{Thm comparison distance to boundary}
Let  $(M^n, g, e^{-f}d\mu_g)$ be a compact smooth metric measure space with smooth nonempty boundary $\p M$.
Let $\vp:[0, R ] \to \R_+$ be a smooth function with $\vp' \geq 0$ and  define $v(x) :=\vp\left(d(x,\p M)\right)$.
\begin{enumerate}
    \item[(i)] Suppose that $\Ric^N_f \geq (N-1)\kappa$ and $H_f \geq \Lambda$ for $\k, \Lambda \in \R$ and $N\in [n,\infty)$. Then $v(x)$ is a viscosity supersolution of \begin{equation*}
    Q[v] =\left(\a (\vp') \vp'' -(N-1) T_{\kappa, \Lambda} \beta(\vp') \vp' \right) \big|_{d(x,\p M)}.
\end{equation*}
\item[(ii)] Suppose that $\Ric_f \geq 0$ and $H_f \geq 0$. Then $v(x)$ is a viscosity supersolution of 
\begin{equation*}
    Q[v] =\left(\a (\vp') \vp'' \right) \big|_{d(x,\p M)}.
\end{equation*}
\end{enumerate}
\end{thm}

Theorem \ref{Thm comparison distance to boundary} generalizes the comparison theorems for the Laplacian \cite{Kasue84} and for the weighted $p$-Laplacian \cite{Sakurai19} since our operator $Q$ is much more general. 
The differential inequalities hold in the classical sense at points where $d(x, \p M)$ is smooth, in the distributional sense if $Q$ is of divergence form and in the viscosity sense for general $Q$.



\begin{proof}[Proof of Theorem \ref{Thm comparison distance to boundary}]
By approximation, it suffices to consider the case $\vp' >0$ on $[0,  R]$. 
By definition of viscosity solutions (see Definition \ref{def viscosity}), it suffices to prove that for any smooth function $\psi$ touching $v$ from below at $x_0 \in M$, i.e., 
\begin{align*}
    \psi(x) \leq v(x) \text{ on } M \text{ with } 
    \psi(x_0) = v(x_0),
\end{align*}
it holds that
\begin{equation*}
    Q^*[\psi](x_0) \leq  \left. \left[\a (\vp') \vp'' -(N-1) T_{\kappa, \Lambda} \beta(\vp') \vp' \right] \right|_{d(x_0,\p M)},
\end{equation*}
where $Q^*$ is the upper-semicontinuous envelope of $Q$ (see Definition 2.2). 

Since the function $d(x, \p M)$ may not be smooth at $x_0$, so we need to replace it by a smooth function $\bar{d}(x)$ defined in a neighborhood $U(x_0)$ of $x_0$ satisfying $\bar{d}(x) \geq d(x, \p M)$ for $x \in U(x_0)$ and $\bar{d}(x_0)=d(x_0, \p M)$. 
The construction is standard (see e.g. \cite[pp. 73-74]{Wu79} or \cite[pp. 1187]{AX19}), which we state below for reader's convenience. 

Since $M$ is compact, there exists $y_0 \in \p M$ such that $d(x_0,y_0)=d(x_0, \p M):=s_0$. Let $\gamma:[0,s_0] \to M$ be the unit speed length-minimizing geodesic with $\gamma(0)=x_0$ and $\gamma(s_0)=y_0$.  
For any vector $X \in \exp^{-1}_{x_0}U(x_0)$, let $X(s) (s\in [0,s_0])$ be the vector field obtained by parallel translating $X$ along $\gamma$, and we decompose it as 
\begin{equation*}
    X(s) =a X^\perp (s) +b \gamma'(s),
\end{equation*}
where $a$ and $b$ are constants along $\gamma$ with $a^2+b^2 =|X(s)|^2$, and $X^\perp(s)$ is a unit parallel vector field along $\gamma$ orthogonal to $\gamma'(s)$. Define 
\begin{equation*}
    W(s)=a \, \eta(s) X^\perp(s) + b\left(1-\frac{s}{s_0} \right)\gamma'(s),
\end{equation*}
where $\eta:[0,s_0] \to \R_+$ is a $C^2$ function to be chosen later. 
Next we define the $n$-parameter family of curves $\gamma_X :[0,s_0] \to M$ such that 
\begin{enumerate}
    \item $\gamma_0 =\gamma$;
    \item $\gamma_X(0) =\exp_{x_0}(W(0))$ and $\gamma_X(s_0) \in \p M$;
    \item $W_l(s)$ is induced by the one-parameter family of curves $l \to \gamma_{lX}(s)$ for $l \in [-l_0,l_0]$ and $s\in [0,s_0]$;
    \item $\gamma_X$ depends smoothly on $X$. 
\end{enumerate}
Finally let $\bar{d}(x)$ be the length of the curve $\gamma_X(x)$ where $x=\exp_{x_0}(X) \in U(x_0)$. 
Then we have
$\bar{d}(x) \geq d(x,\p M)$ on $U(x_0)$, $\bar{d}(x_0)=d(x_0, \p M)$. Recall the  first and second variation formulas:
\begin{equation*}
\n \bar{d}(x_0) =-\gamma'(0) 
\end{equation*}
and
\begin{equation*}
    \n^2 \bar{d} (X,X)= -a^2 \eta(s_0)^2 \operatorname{II}(X^\perp(s_0), X^\perp(s_0)) +a^2 \int_0^{s_0} \left((\eta')^2 -\eta^2 R(X^\perp, \gamma',X^\perp, \gamma') \right) ds
\end{equation*}
where $\operatorname{II}$ denotes the second fundamental form of $\p M$ at $y_0$. Then for an orthonormal frame $\{e_i(s)\}_{i=1}^n$ along $\gamma$ with $e_n(s) =\gamma'(s)$ we have
\begin{equation}\label{1st bard}
\n \bar{d}(x_0) =-e_n(0),
\end{equation}
\begin{equation}\label{2nd bardn}
\n^2 \bar{d} (e_n(0),e_n(0))=0,
\end{equation}
and for $1\le i \le n-1$
\begin{equation}\label{2nd bardi}
    \n^2 \bar{d} (e_i(0),e_i(0))= - \eta(s_0)^2 \operatorname{II}(e_i(s_0), e_i(s_0)) + \int_0^{s_0} (\eta')^2 -\eta^2 R(e_i, e_n,e_i, e_n)\, ds.
\end{equation}
Since the function $\psi(x) -\vp\left(d(x, \p M)\right)$ attains its maximum at $x_0$ and $\vp' >0$, it follows that the function $\psi(x) -\vp(\bar{d}(x))$ attains a local maximum at $x_0$. The first and second derivative tests yield
$$ \n  \psi (x_0) =-\vp' e_n (0), \quad \psi_{nn}(x_0)  \leq  \vp'',$$
    and
    $$
    \psi_{ii}(x_0)  \leq \vp' \n^2\bar{d} \left(e_i(0), e_i(0)\right)
$$
for $1 \le i \leq n-1$, where we used (\ref{1st bard}) and (\ref{2nd bardn}).
Here and below the derivatives of $\vp$ are all evaluated at $s_0=d(x_0, \p M)$. 
It then follows from (\ref{2nd bardi}) that
\begin{equation*}
\sum_{i=1}^{n-1} \n^2 \bar{d}\left(e_i(0), e_i(0)\right) =- \eta(s_0)^2 H(y_0) + \int_0^{s_0}(n-1)(\eta')^2 -\eta^2 \Ric(e_n, e_n)  \, ds,
\end{equation*}
and then we have 
\begin{eqnarray}\label{eq 3.1} 
    && Q^*[\psi](x_0)=Q[\psi](x_0) \nonumber \\
    &=&\a (\vp')\psi_{nn} +\b(\vp') \sum_{i=1}^{n-1} \psi_{ii} + \b(\vp')\vp' \langle \n f(x_0), e_n(0) \rangle \nonumber\\
    &\leq& \a (\vp')\vp'' +\b(\vp')\vp' \left( \int_0^{s_0} (n-1)(\eta')^2 -\eta^2 \Ric(e_n, e_n) \, ds \right)\nonumber\\
    &&+\b(\vp')\vp'\left(  -\eta(s_0)^2 H(y_0) +\langle \n f(x_0), e_n(0) \rangle\right) 
\end{eqnarray}
where $H$ denote the mean curvature of $\p M$. 
We estimate using $\Ric^N_f \geq (N-1) \kappa$ that 
\begin{eqnarray*}
&& \int_0^{s_0} \left((n-1)(\eta')^2 -\eta^2 \Ric(e_n, e_n) \right) ds \\
& \leq & (N-1) \int_0^{s_0} (\eta')^2\, ds -(N-n) \int_0^{s_0} (\eta')^2\, ds -(N-1)\k \int_0^{s_0} \eta^2 ds \\
&&   + \int_0^{s_0} \eta^2 \n^2 f (\gamma',\gamma') ds -\frac{1}{N-n}\int_0^{s_0} \eta^2 \n f \otimes \n f (\gamma', \gamma')ds \\
&=&  (N-1) \int_0^{s_0} (\eta')^2 \, ds-(N-n) \int_0^{s_0} (\eta')^2 \, ds-(N-1)\k \int_0^{s_0} \eta^2 ds \\
&& + \left. \eta^2 (f\circ \gamma )' \right|_0^{s_0} - 2 \int_0^{s_0} \eta \, \eta' (f \circ \gamma)' ds  -\frac{1}{N-n} \int_0^{s_0} \eta^2 ((f\circ \gamma)' )^2 ds  \\
&=&  (N-1) \int_0^{s_0} (\eta')^2\, ds -(N-1)\k \int_0^{s_0} \eta^2 ds +\left. \eta^2 (f\circ \gamma )' \right|_0^{s_0} \\
&& -\int_0^{s_0} \frac{\eta^2}{N-n}\left((N-n)\frac{\eta'}{\eta} + (f\circ \gamma)' \right) ^2\, ds \\
& \leq & (N-1) \int_0^{s_0} (\eta')^2\, ds -(N-1)\k \int_0^{s_0} \eta^2 ds +\left. \eta^2 (f\circ \gamma )' \right|_0^{s_0} 
\end{eqnarray*}

Using $H_f \geq (N-1)\Lambda$ and choosing $\eta(s)=C_{\kappa, \Lambda}(s_0-s) /C_{\kappa, \Lambda}(s_0)$ with $C_{\kappa, \Lambda}$ defined in \eqref{C def}, we calculate that
\begin{eqnarray*}
&& -\eta(s_0)^2 H(y_0) +\langle \n f(x_0), e_n(0)\rangle+ \int_0^{s_0} (n-1)(\eta')^2 -\eta^2 \Ric(e_n, e_n) \, ds \\
& \leq &  -\frac{1}{C^2_{\k, \Lambda}(s_0)}H(y_0) +(N-1) \int_0^{s_0} (\eta')^2 \, ds-(N-1)\k \int_0^{s_0} \eta^2 ds + \frac{1}{C^2_{\k, \Lambda}(s_0)} (f\circ \gamma )'(s_0)  \\
&=& -\frac{1}{C^2_{\k, \Lambda}(s_0)}H_f(y_0)-(N-1)\kappa \int_0^{s_0} \eta^2 ds \\
&&+(N-1) \left(  \eta(s_0)\eta'({s_0}) -\eta(0)\eta'(0)  -\int_0^{s_0} \eta(s) \eta''(s) ds \right) \\
&\leq& -\frac{N-1}{C^2_{\k, \Lambda}(s_0)}\Lambda + (N-1)\eta'(s_0) \eta(s_0)-(N-1)\eta(0)\eta'(0)\\
&=&-(N-1) T_{\kappa, \Lambda},
\end{eqnarray*}
where we used $C_{\k, \Lambda}'(0)=-\Lambda$.
Thus, 
\begin{eqnarray*}
     Q^*[\psi](x_0) \leq \left[\a (\vp')\vp'' -(N-1) T_{\kappa, \Lambda} \b(\vp')\vp' \right]_{d(x_0, \p M)}.
\end{eqnarray*}

For the $N=\infty$ case, we choose $\eta(s)=1$ in \eqref{eq 3.1} and compute that 
\begin{eqnarray*}\label{} \nonumber
     Q^*[\psi](x_0) &\leq& \a (\vp')\vp'' +\b(\vp')\vp' \left( \int_0^{s_0} -\Ric(e_n, e_n) \, ds - H(y_0) +\langle \n f(x_0), e_n(0) \rangle \right)\\
     & \leq & \a (\vp')\vp'' + \int_0^{s_0} (f \circ \gamma )''(s) ds -H_f(y_0)  - (f \circ \gamma)'(s_0) +(f \circ \gamma)'(0) \\
     &\leq & \a (\vp')\vp''
\end{eqnarray*}
where we have used $\Ric + \n^2 f \geq 0$ in the second inequality and $H_f \geq 0$ in the last inequality.
The proof is complete. 
\end{proof}

\section{Comparison theorems for the second derivatives of $d(x,y)$}

In this section, we prove a comparison theorem for the second derivatives of $d(x,y)$, which generalize \cite[Theorem 3]{AC13} on Riemannian manifolds to smooth metric measure spaces. 

Let $(x_0,y_0)\in M\times M\setminus\{(x,x) : x\in M\}$ and $d(x_0, y_0)=s_0$.
Let $\gamma_0:[0,s_0]\rightarrow M$ be a unit minimizing geodesic joining $x_0$ and $y_0$ with $\gamma_0(0)=x_0$ and $\gamma_0(s_0)=y_0$. We choose Fermi coordinates $\{e_i(s)\}$ along $\gamma_0$ with $e_n(s)=\gamma_0'(s)$. Note that the distance function $d(x,y)$ may not be smooth at $(x_0,y_0)$, so one cannot apply the maximum principle for semicontinuous functions on manifolds directly. To overcome this, we proceed as in \cite[pages 561-562]{LW17} and replace $d(x,y)$ by a smooth function $\rho(x,y)$, which is defined as follows:

\begin{definition}[Modified distance function]\label{def-rho}
Let  $U(x_0,y_0)\subset M\times M\setminus\{(x,x) : x\in M\}$ be a neighborhood of $(x_0, y_0)$. 
Define variation fields $V_i(s)$ along $\gamma_0(s)$ by
$V_i(s)=\eta(s) e_i(s)$ for $1\le i\le n-1$, 
and $V_n(s)=e_n(s)$, where $\eta(s)$ is a smooth function to be chosen later. We then define a smooth function $\rho(x,y)$ in $U(x_0,y_0)$ to be the length of the curve $$\exp_{\gamma_0(s)}{\sum\limits_{i=1}^n\Big((1-\frac{s}{s_0})b_i(x)+\frac{s}{s_0}c_i(y)\Big)V_i(s)}$$
for $s\in [0,s_0]$, where $b_i(x)$ and $c_i(y)$ are so defined that
$$
x=\exp_{x_0}\left(\sum_{i=1}^n b_i(x) e_i(0)\right) \text{\quad and \quad } y=\exp_{y_0}\left(\sum_{i=1}^n c_i(y) e_i(s_0)\right).
$$
\end{definition}
For $\rho(x,y)$ defined above, it is well known from the standard variation formulas of arc-length that  
\begin{lemma}[Variation formulas]
The first variation formula gives
\begin{equation}\label{1st}
    \n \rho(x_0,y_0)=\big(-e_n(0),e_n(s_0)\big).
\end{equation}
The second variation formula gives
\begin{equation}\label{2nd1}
   \n^2 \rho \Big(\big(e_n(0),\pm e_n(s_0)\big),\big(e_n(0),\pm e_n(s_0)\big)\Big)=0,
   \end{equation}
and for $1\le i\le n-1$
\begin{equation}\label{2nd2}
    \n^2 \rho \Big(\big(e_i(0),e_i(s_0)\big),\big(e_i(0),e_i(s_0)\big)\Big)=\int_0^{s_0}(\eta')^2-\eta^2 R\big(e_i,e_n,e_i,e_n\big) \, ds
\end{equation}
at $(x_0, y_0)$.
\end{lemma}


\begin{thm}\label{Thm comparison distance}
Let $(M^n, g, e^{-f} d\mu_g)$ be a compact smooth metric measure space with diameter $D$. Let $\vp: [0, \frac{D}{2}] \to \R$ be a smooth nondecreasing function with $\vp'\ge 0$, and define $v(x,y):=2\vp\left(\frac{d(x,y)}{2}\right)$.
\begin{enumerate}
    \item[(i)] Suppose $\Ric^N_f \geq (N-1)\kappa$  for $N \in [n,\infty)$. Then on the set $ (M\times M) \setminus \{(x,x) : x\in M\}$, the function $v(x,y)$ is a viscosity supersolution of
\begin{equation*}
    L[\nabla^2 v, \nabla v] = 2 \Big(\a(\vp') \vp '' -(N-1)T_{\k,0} \b(\vp')\vp' \Big)\big|_{ s=\frac{d(x,y)}{2}}.
\end{equation*}
\item[(ii)] Suppose $\Ric_f\geq \k$  for $\k \in \R$. Then on the set $ (M\times M) \setminus \{(x,x) : x\in M\}$, the function $v(x,y)$ is a viscosity supersolution of
\begin{equation*}
    L[\n^2 v, \n v] = 2 \Big(\a(\vp') \vp '' -\k s \b(\vp')\vp' \Big)\Big|_{s=\frac{d(x,y)}{2}}.
\end{equation*}
\end{enumerate}
Here the operator $L$ is defined by
\begin{align*}
     L[B, w] =\inf & \left\{  \tr(AB)-\b(|w|)\langle\nabla (f(x)+f(y)), w\rangle
     : A \in \text{Sym}^2(T^*_{(x,y)}M\times M), \right. \\
     & \text{   }\left. A \geq 0,  A|_{T^*_x M} =a (w|_{T_x M}), A|_{T^*_y M} =a (w|_{T_y M})  \right\}
\end{align*}
for any $B \in \text{Sym}^2(T_{(x,y)} M\times M)$ and $w \in T^*_{(x,y)}M \times M$. Where $a(w)$ is defined by
$$a(w)(\xi,\xi)=\a(|w|)\frac{(w\cdot \xi)^2}{|w|^2}+\b(|w|)(|\xi|^2-\frac{(w\cdot \xi)^2}{|w|^2}).$$
\end{thm}

\begin{proof}
The case $N=n$ has been proved in \cite{AC13}. The proof here is a slight modification of the proof of Theorem 3 in \cite{AC13} and we include it for the reader's convenience.

By approximation it sufﬁces to consider the case where $\vp'$ is strictly positive. For any $x_0,y_0 \in M$ with $x_0\neq y_0$, it suffices to show that any smooth function $\psi(x,y)$ satisfies
$$
\psi(x, y)\le v(x,y)
$$
in a neighborhood $U(x_0,y_0)$ of $(x_0,y_0) $ with equality at $(x_0, y_0)$, it holds 
\begin{equation}\label{2.1}
    L[\nabla^2 \psi,\nabla \psi](x_0, y_0)\le 2 \ \Big(\a(\vp') \vp '' -(N-1)T_{\k,0} \b(\vp')\vp' \Big)\Big|_{s=\frac{s_0}{2}}
\end{equation}
for $n\le N<\infty$, and
\begin{equation}\label{2.2}
    L[\nabla^2 \psi,\nabla \psi](x_0, y_0)\le 2 \Big(\a(\vp') \vp '' -\k s \b(\vp')\vp' \Big)\Big|_{s=\frac{s_0}{2}}
\end{equation}
for $N=\infty$. Where $s_0=d(x_0, y_0)$.

From the definition of $\rho(x,y)$,  we see clearly  that $d(x,y)\le \rho(x,y)$ and $d(x_0,y_0)=\rho(x_0,y_0)$. Since
 $\vp'>0$,  we have that
\begin{equation}\label{2.6}
\psi(x,y)\le v(x,y) \le 2\vp(\frac{\rho(x,y)}{2})
\end{equation}
in $U(x_0,y_0)$ with equality at $(x_0, y_0)$. Then the first derivative of $\psi$ at $(x_0, y_0)$ yields
$$
\n_x \psi=- \vp' e_n(0) \text{\quad and\quad} \n_y \psi= \vp' e_n(s_0),
$$
where we used the first variation formula (\ref{1st}). Here and below the derivatives of $\vp$ are all evaluated at $\frac{s_0}{2}$. 
From the definition of $a$ we deduce
$$
a(\n \psi|_{T_{x_0}M})=\a(\vp')e_n(0)\otimes e_n(0)+\b(\vp')\sum_{i=1}^{n-1}e_i(0)\otimes e_i(0)
$$
and
$$
a(\n \psi|_{T_{y_0}M})=\a(\vp')e_n(s_0)\otimes e_n(s_0)+\b(\vp')\sum_{i=1}^{n-1}e_i(s_0)\otimes e_i(s_0).
$$

To prove inequality (\ref{2.1}) and (\ref{2.2}), we choose $A$ as follows
$$
A=\a(\vp')\big(e_n(0),-e_n(s_0)\big)\otimes \big(e_n(0),-e_n(s_0)\big)+\b(\vp')\sum_{i=1}^{n-1}\big(e_i(0), e_i(s_0)\big)\otimes\big (e_i(0),e_i(s_0)\big)
$$
which is clearly nonnegative, and agrees with $a$ on $T_{x_0}M$ and $T_{y_0}M$ as required. This choice gives
\begin{eqnarray}\label{2.7}
  \operatorname{tr}(A\n^2\psi)
   &=&\a(\vp')\n^2\psi\big((e_n(0), -e_n(s_0)),(e_n(0), -e_n(s_0))\big)\nonumber\\
   &&+\sum_{i=1}^{n-1}\beta(\vp')\n^2\psi\big((e_i(0), e_i(s_0)),(e_i(0), e_i(s_0))\big).  
\end{eqnarray}
 Now we estimate the terms involving second derivatives of $\psi$. Recall from (\ref{2.6}) that $\psi(x,y)-2\vp(\frac{\rho(x,y)}{2})$ attains a local maximum at $(x_0, y_0)$, then the second derivatives at $(x_0, y_0)$ yields
 \begin{eqnarray}\label{2.9}
&&\n^2 \psi \big((e_n(0), -e_n(s_0)),(e_n(0), -e_n(s_0))\big)\nonumber\\
&\le& 2\n^2\vp(\frac{\rho(x,y)}{2})\big((e_n(0), -e_n(s_0)),(e_n(0), -e_n(s_0))\big)\nonumber\\
&=&\frac{d}{dt}\Big|_{t=0}2\vp(\frac{s_0}{2}-t)\nonumber\\
&=&2\vp''(\frac{s_0}{2})
\end{eqnarray}
and for $1\le i\le n-1$
\begin{eqnarray}\label{2.8}
&&\n^2 \psi \big((e_i(0), e_i(s_0)),(e_i(0), e_i(s_0))\big)\nonumber\\
&\le& 2\n^2\vp(\frac{\rho(x,y)}{2})\big((e_i(0), e_i(s_0)),(e_i(0), e_i(s_0))\big)\nonumber\\
&=&\vp'(\frac{s_0}{2})\n^2 \rho\big((e_i(0),e_i(s_0)),((e_i(0),e_i(s_0))\big)\nonumber\\
&=&\vp'(\frac{s_0}{2})\int_0^{s_0}(\eta')^2-\eta^2 R(e_i,e_n,e_i,e_n) \, ds
\end{eqnarray}
where we used  the  variation formulas (\ref{1st}), (\ref{2nd1}) and (\ref{2nd2}). 
Substituting (\ref{2.9}) and (\ref{2.8}) to (\ref{2.7}), we obtain
\begin{equation*}
\operatorname{tr}(A\n^2\psi)\le 2 \a(\vp')\vp''+\beta(\vp')\vp'\int_0^{s_0}(n-1)(\eta')^2-\eta^2 \operatorname {Ric}(e_n,e_n) \, ds.
\end{equation*}
Therefore, using the definition of $L$ we have
\begin{eqnarray}\label{2.10}
&&L[\n^2 \psi, \n \psi](x_0, y_0)\nonumber\\
&\le& \operatorname{tr}(A\n^2\psi)-\b(\vp')\langle \n (f(x_0)+f(y_0)) , \n \psi\rangle\nonumber\\
&\le&2 \a(\vp')\vp''+\beta(\vp')\vp'\int_0^{s_0}(n-1)(\eta')^2-\eta^2 \operatorname {Ric}(e_n,e_n) \, ds\nonumber\\
&&-\b(\vp')\langle \n f(x_0) , \n_x \psi\rangle-\b(\vp')\langle \n f(y_0) , \n_y \psi\rangle.
\end{eqnarray}

Now we prove inequality (\ref{2.1}) and (\ref{2.2}).\\
\textbf{Case 1.} $n<N<\infty$.
Choose $$\eta(s)=\frac{C_{\k, 0}(s-\frac{s_0}{2})}{C_{\k,0}(\frac{s_0}{2})}$$in definition of $\rho(x,y)$, and we estimate  using $\operatorname{Ric}_f^N\ge (N-1)\k$ that
\begin{eqnarray}\label{2.11}
&&\int_0^{s_0}(n-1)(\eta')^2 -\eta^2 \Ric(e_n, e_n) \,ds \nonumber\\
& \leq & (N-1) \int_0^{s_0} (\eta')^2\,ds -(N-n) \int_0^{s_0} (\eta')^2 \,ds-(N-1)\k \int_0^{s_0} \eta^2 \,ds \nonumber\\
&&   + \int_0^{s_0} \eta^2 \n^2 f (e_n,e_n) \, ds -\frac{1}{N-n}\int_0^{s_0} \eta^2 \n f \otimes \n f (e_n, e_n)\, ds \nonumber\\
&=&  (N-1) \int_0^{s_0} (\eta')^2\,ds -(N-n) \int_0^{s_0} (\eta')^2 \,ds-(N-1)\k \int_0^{s_0} \eta^2 \,ds \nonumber\\
&& + \left. \eta^2 (f\circ \gamma_0 )' \right|_0^{s_0} - 2 \int_0^{s_0} \eta \, \eta' (f \circ \gamma_0)' ds  -\frac{1}{N-n} \int_0^{s_0} \eta^2 ((f\circ \gamma_0)' )^2 ds  \nonumber\\
&\le& (N-1) \int_0^{s_0} (\eta')^2 \, ds-(N-1)\k \int_0^{s_0} \eta^2 \, ds\nonumber\\
&&+\langle \n f(y_0), e_n(s_0) \rangle-\langle \n f(x_0), e_n(0) \rangle,
\end{eqnarray}
where we used $(N-n)(\eta')^2+2\eta \eta' (f\circ \gamma_0)'+\frac{1}{N-n}\eta^2((f\circ \gamma_0)')^2\ge0$ in the last inequality. Since
\begin{eqnarray*}
\int_0^{s_0} (\eta')^2 -\k \eta^2 \, ds=\eta \eta'|_0^{s_0}-\int_0^{s_0} \eta''\eta +\k \eta^2 \, ds
=-2T_{\k, 0}(\frac{s_0}{2}),
\end{eqnarray*}
then we deduce from (\ref{2.11}) that
\begin{eqnarray}\label{eta1}
&&\int_0^{s_0}(n-1)(\eta')^2 -\eta^2 \Ric(e_n, e_n) \, ds\nonumber\\
&\le& -2(N-1)T_{\k, 0}(\frac{s_0}{2})+\langle \n f(y_0), e_n(s_0) \rangle-\langle \n f(x_0), e_n(0) \rangle
\end{eqnarray}
Combining (\ref{2.10}) and ({\ref{eta1}}) together, we obtain \begin{eqnarray*}
 L[\nabla^2 \psi,\nabla \psi](x_0, y_0)\le2\a(\vp')\vp''-2(N-1)T_{\k,0}\b(\vp')\vp'
\end{eqnarray*}
 where we used $\n_x \psi=- \vp'(\frac{s_0}{2}) e_n(0)$ and $  \n_y \psi= \vp' (\frac{s_0}{2})e_n(s_0)$ at $(x_0, y_0)$. Thus we proved  (\ref{2.1}).

\textbf{Case 2.} $ N=\infty$.
Choose $\eta=1$, and then we estimate using $\operatorname{Ric}+\n^2 f\ge \k$ that
\begin{eqnarray}\label{eta2}
&&\int_0^{s_0}(n-1)(\eta')^2-\eta^2 \operatorname {Ric}(e_n,e_n) \, ds\nonumber\\
&=&\int_0^{s_0}-\operatorname {Ric}(e_n,e_n) \, ds
\nonumber\\
&\le&\int_{0}^{s_0}-\k +(f\circ \gamma_0(s))'' \, ds\nonumber\\
&=&-\k s_0+\langle \n f(y_0), e_n(s_0) \rangle-\langle \n f(x_0), e_n(0) \rangle.
\end{eqnarray}
Substituting inequality (\ref{eta2}) to (\ref{2.10}) we get at $(x_0, y_0)$
\begin{eqnarray*}
 L[\nabla^2 \psi,\nabla \psi]&\le&\operatorname{tr}(A\n^2\psi)-\b(\vp')\langle\nabla (f(x)+f(y)), \n \psi\rangle\\
 &\le&2\a(\vp')\vp''-\k s_0 \b(\vp')\vp'\\
 &=&2 \Big(\a(\vp') \vp '' -\k s \b(\vp')\vp' \Big)\Big|_{s=\frac{s_0}{2}},
\end{eqnarray*}
which proves (\ref{2.2}).
\end{proof}


\section{Modulus of Continuity Estimates for Neumann Boundary Condition}
The goal of this section is to extend the results in \cite{AC13} from Riemannian manifolds to smooth metric measure spaces, as well as from smooth solutions to viscosity solutions. 

\begin{thm}\label{thm Neumann}
Let $(M^n, g, e^{-f}d\mu_g)$ be a compact smooth metric measure space with diameter $D$ (possibly with smooth strictly convex boundary). Let $u:M\times [0,T)\rightarrow \mathbb{R}$ be a viscosity solution of $u_t=Q[u]$
(with Neumann boundary conditions if $\p M \neq \emptyset$), where $Q[u]$ is defined in \eqref{Q def}. 
Assume $\Ric^N_f \geq (N-1) \k$ for $N\in [n,\infty)$ and $\k \in \R$, or $\Ric_f\geq \k$ for $\k \in \R$. 
Let $\vp:[0, D/2] \times \R_+ \to \R_+$ be a smooth function satisfying 
\begin{enumerate}
\item $\vp(s,0) =\vp_0(s)$ for each $s \in [0,D/2]$; 
     \item $\vp_t \geq \a(\vp')\vp'' -(N-1) T_{\kappa, 0} \b(\vp')\vp'$ if $N\in[n, \infty)$, or \\
     $\vp_t \geq \a(\vp')\vp'' -\k s \b(\vp')\vp'$ if $N= \infty$;
    \item $\vp'\geq 0$ on $[0, D/2] \times \R_+$;
    \item $\vp_0$ is a modulus of continuity of $u(x, 0)$. 
\end{enumerate}
Then $\vp(s,t)$ is a modulus of continuity for $u(x,t)$ for each $t \in [0,T)$, i.e., 
\begin{equation*}
    u(y,t)-u(x,t) -2\vp\left(\frac{d(x,y)}{2},t \right) \leq 0
\end{equation*}
for all $x,y \in M$ and $t\in [0,T)$.
\end{thm}


\begin{proof}[Proof of Theorem \ref{thm Neumann}]
On the product manifold $M \times M \times [0,T)$, define an evolving quantity $Z_{\eps}(x,y,t)$ by
\begin{equation*}
    Z_{\eps}(x,y,t) =u(y,t) -u(x,t) -2\vp \left(\frac{d(x,y)}{2},t \right) -\frac{\eps}{T-t},
\end{equation*}
for any small positive $\eps$.
Note that we have $Z_{\eps}(x,y,0) <0$ since we assumed that $\vp_0$ is a modulus of continuity for $u$ at $t = 0$. Also observe that $Z_{\eps}(x,x,t) <0$ for all $x\in M$ and $t\in [0,T)$. Thus, if $Z_{\eps}$ ever becomes positive, there exists a time $t_0 > 0$ and points $x_0 \neq  y_0$ in
$M$ such that $Z_{\eps}$ attains its maximum at $(x_0,y_0,t_0)$. Notice that the Neumann condition, convexity of the boundary $\p M$, and  the positivity of $\vp'$ guarantees that both $x_0$ and $y_0$ are in $M$ if $(M^n, g)$ has non-empty boundary.

Since the distance function $d(x,y)$ may not be smooth, we replace $d(x,y)$ by the smooth function $\rho(x,y)$ defined in Definition \ref{def-rho}, and the function 
$$u(y,t)-u(x, t)-2\vp\left(\frac{\rho(x,y)}{2},t\right)-\frac{\eps}{T-t}$$
has a local maximum at $(x_0,y_0,t_0)$ by the monotonically increasing  of $\vp$.
Now we can apply the parabolic version maximum principle (Theorem \ref{max prin}) for semicontinuous functions on manifolds to conclude that:
for each $\lambda >0$, there exist symmetric tensors $X, Y$ such that
\begin{eqnarray}\label{1stx}
   (b_1, \n_y \psi (x_0, y_0, t_0), X) &\in \overline{\mathcal{P}}^{2,+} u(y_0,t_0),\\
 (-b_2, - \n_x \psi (x_0, y_0, t_0), Y) &\in \overline{\mathcal{P}}^{2,-} u(x_0,t_0),
\end{eqnarray}
\begin{eqnarray}\label{1stt}
 b_1+b_2= \psi_t (x_0, y_0, t_0)=2 \vp_t(\frac{s_0}{2},t_0)+\frac{\eps}{(T-t_0)^2},
 \end{eqnarray}
 and
\begin{eqnarray}\label{2ndx}
    \begin{pmatrix}
   X & 0 \\
   0 & -Y
   \end{pmatrix}
   \leq S+\lambda S^2,
  \end{eqnarray}
 where $\psi(x,y,t)=2\vp(\frac{\rho(x,y)}{2},t)+\frac{\eps}{T-t}$, $S=\n^2 \psi(x_0,y_0,t_0)$, and $s_0=d(x_0,y_0)$.

Using the first derivative formula (\ref{1st}) of $\rho$, we have 
\begin{equation}\label{der psi}
\n_x \psi (x_0, y_0, t_0)=-\vp'(\frac{s_0}{2}, t_0)e_n(0)
\text{\quad and \quad}
\n_y \psi (x_0, y_0, t_0)=\vp'(\frac{s_0}{2}, t_0)e_n(s_0). 
\end{equation}
Since $u$ is both a subsolution and a supersolution of (\ref{Q def}), then (\ref{1stx}) yields
\begin{equation}\label{b1}
b_1\le \operatorname{tr}(A(\vp')X)-\b(\vp')\vp'\langle \n f(y_0), e_n(s_0)\rangle
\end{equation}
and
\begin{equation}\label{b2}
-b_2\ge \operatorname{tr}(A(\vp')Y)-\b(\vp')\vp'\langle \n f(x_0), e_n(0)\rangle 
\end{equation}
where  $A$ is a diagonal matrix defined by
$$
A=\left(
    \begin{array}{cccc}
      \beta(\vp') & \cdots & 0 & 0 \\
     \vdots & \vdots & \vdots & \vdots \\
     0 & \cdots &\beta(\vp') & 0 \\
      0 & \cdots & 0 & \alpha(\vp') \\
    \end{array}
  \right)
$$
and we have used equality (\ref{der psi}) and $\vp'>0$.

Set
$$C=\left(
    \begin{array}{cccc}
      \b(\vp') & \cdots & 0 & 0 \\
     \vdots & \vdots & \vdots & \vdots \\
     0 & \cdots &\b(\vp') & 0 \\
      0 & \cdots & 0 & -\a(\vp') \\
    \end{array}
  \right),$$
and then $\left(
                           \begin{array}{cc}
                             A & C \\
                            C & A \\
                           \end{array}
                         \right)\ge 0$.
Substituting inequality (\ref{b1}) and (\ref{b2}) to (\ref{1stt}), we have 
\begin{eqnarray}\label{a5}
&&2\vp_t(\frac{s_0}{2}, t_0)+\frac{\eps}{(T-t_0)^2}\\
&=& b_1+b_2\nonumber\\
&\le& \operatorname{tr}\left[\left(
                                                        \begin{array}{cc}
                                                          A & C \\
                                                         C & A \\
                                                        \end{array}
                                                      \right)
\left(
\begin{array}{cc}
 X & 0 \\
  0 & -Y \\
   \end{array}
   \right)
\right] \nonumber\\
&& -\b(\vp')\vp'\langle \n f(y_0), e_n(s_0)\rangle+\b(\vp')\vp'\langle \n f(x_0), e_n(0)\rangle\nonumber\\
&\le&\operatorname{tr}\left[\left(
                                                        \begin{array}{cc}
                                                          A & C \\
                                                         C & A \\
                                                        \end{array}
                                                      \right)S
\right]
+\lambda\operatorname{tr}\left[\left(
                                                        \begin{array}{cc}
                                                          A & C \\
                                                         C & A \\
                                                        \end{array}
                                                      \right)S^2
\right]\nonumber\\
&&-\b(\vp')\vp'\big(\langle \n f(y_0), e_n(s_0)\rangle-\langle \n f(x_0), e_n(0) \rangle\big),\nonumber
\end{eqnarray}
where we used inequality (\ref{2ndx}). 
Direct calculation shows 
\begin{eqnarray}\label{a4}
&&\operatorname{tr}\left[\left(
                                                        \begin{array}{cc}
                                                          A & C \\
                                                         C & A \\
                                                        \end{array}
                                                      \right)S
\right]\nonumber\\
&=&\a(\vp')\n^2\psi\big((e_n(0), -e_n(s_0)),(e_n(0), -e_n(s_0))\big)\nonumber\\
   &&+\sum_{i=1}^{n-1}\beta(\vp')\n^2\psi\big((e_i(0), e_i(s_0)),(e_i(0), e_i(s_0))\big)\nonumber\\
   &=&2\a(\vp')\vp''+\beta(\vp')\vp'\sum_{i=1}^{n-1}\n^2\rho\big((e_i(0), e_i(s_0)),(e_i(0), e_i(s_0))\big)\nonumber\\
   &=&2\a(\vp')\vp''+\beta(\vp')\vp'\int_0^{s_0}(n-1)(\eta')^2-\eta^2 \operatorname {Ric}(\gamma_0',\gamma_0') \, ds
\end{eqnarray}
where we used the variation formulas (\ref{1st}), (\ref{2nd1}) and (\ref{2nd2}) for $\rho(x,y)$.

If $\Ric_f^N \ge(N-1)\k$, we choose $\eta=\frac{C_{\k, 0}(s-\frac{s_0}{2})}{C_{\k,0}(\frac{s_0}{2})}$ in (\ref{a4}),  and  using (\ref{eta1}) we obtain
\begin{eqnarray*}
\operatorname{tr}\left[\left(
                                                        \begin{array}{cc}
                                                          A & C \\
                                                         C & A \\
                                                        \end{array}
                                                      \right)S
\right]&\le& 2\a(\vp')\vp''-2(N-1)T_{\k, 0}(\frac{s_0}{2})\beta(\vp')\vp' \\
&& +\beta(\vp')\vp'(\langle \n f(y_0), e_n(s_0) \rangle-\langle \n f(x_0), e_n(0) \rangle),
\end{eqnarray*}
and then (\ref{a5}) yields
\begin{eqnarray}\label{a6}
&&2\vp_t(\frac{s_0}{2},t_0)+\frac{\eps}{(T-t_0)^2}\nonumber\\
&\le&2 \ \Big(\a(\vp') \vp '' -(N-1)T_{\k,0} \vp'\b(\vp') \Big)\Big|_{s=\frac{s_0}{2},t=t_0}+\lambda\operatorname{tr}\left[\left(
                                                        \begin{array}{cc}
                                                          A & C \\
                                                         C & A \\
                                                        \end{array}
                                                      \right)S^2
\right].
\end{eqnarray}

If $\Ric_f\ge \k$, we choose $\eta=1$, and substituting  (\ref{eta2}) to equality (\ref{a4}) we obtain 
\begin{eqnarray*}
\operatorname{tr}\left[\left(
                                                        \begin{array}{cc}
                                                          A & C \\
                                                         C & A \\
                                                        \end{array}
                                                      \right)S
\right]
&\le& 2\a(\vp')\vp''-\k s_0\beta(\vp')\vp'\\
&&+\beta(\vp')\vp'(\langle \n f(y_0), e_n(s_0) \rangle-\langle \n f(x_0), e_n(0) \rangle),
\end{eqnarray*}
and then (\ref{a5}) yields
\begin{eqnarray}\label{a7}
&&2\vp_t(\frac{s_0}{2},t_0)+\frac{\eps}{(T-t_0)^2}\nonumber\\
&\le&2\ \Big(\a(\vp') \vp '' -\k s \b(\vp')\vp' \Big)\Big|_{s=\frac{s_0}{2},t=t_0}+\lambda\operatorname{tr}\left[\left(
                                                        \begin{array}{cc}
                                                          A & C \\
                                                         C & A \\
                                                        \end{array}
                                                      \right)S^2
\right].
\end{eqnarray}
Since  $\lambda$ is arbitrary, we get the contradictions with the assumption 2 by letting   $\lambda\rightarrow 0$ in (\ref{a6}) and (\ref{a7}). 

Therefore we conclude that $Z_{\eps}(x,y,t)\le 0$ for $t\in [0, T)$. Letting $\eps\rightarrow 0^+$, we finish the proof of Theorem \ref{thm Neumann}.
\end{proof}

As an application of Theorem \ref{thm Neumann}, we obtain optimal lower bound on the smallest positive eigenvalue of the weighted $p$-Laplacian with $1<p\leq 2$ on a smooth metric measure space. 
\begin{thm}\label{thm Neumann Eigenvalue}
Fix $1<p\leq 2$. Let $(M^n, g, e^{-f}d\mu_g)$ be a compact smooth metric measure space with diameter $D$ (possibly with smooth strictly convex boundary). 
Let $\l_{1,p}$ be the first nonzero closed or Neuman eigenvalue of the weighted $p$-Laplacian. 
\begin{enumerate}
    \item If $\Ric^N_f \geq (N-1) \k$ for $N\in [n,\infty)$ and $\k \in \R$, then we have 
\begin{equation*}
    \l_{1,p} \geq \bar{\l}_{1,p}
\end{equation*}
where $\bar{\l}_{1,p}$ is the first nonzero Neuman eigenvalue of the one-dimensional problem 
\begin{equation*}
    (p-1)|\vp'|^{p-1}\vp'' -(N-1) T_{\k, 0} |\vp'|^{p-2}\vp' =-\l |\vp|^{p-2}\vp
\end{equation*}
on the interval $[-D/2, D/2]$. 
\item If $N=\infty$ and $\Ric_f \geq \kappa$ for $\kappa \in \R$, then we have 
\begin{equation*}
    \l_{1,p} \geq \bar{\mu}_{1,p}
\end{equation*}
where $\bar{\mu}_{1,p}$ is the first nonzero Neuman eigenvalue of the one-dimensional problem 
\begin{equation*}
    (p-1)|\vp'|^{p-1}\vp'' -\k t |\vp'|^{p-2}\vp' =-\l |\vp|^{p-2}\vp
\end{equation*}
on the interval $[0, D]$. 
\end{enumerate}
\end{thm}
\begin{proof}
The proof is a slight modification of \cite[Section 8]{Andrewssurvey15} or \cite[Section 2]{LW19eigenvalue}, so we omit the details here. 
\end{proof}

\begin{remark}
We emphasis that Theorem \ref{thm Neumann Eigenvalue} holds for $2<p<\infty$ as well, as shown by Koerber \cite{Koerber18} for the $\kappa=0$ case and by the second author \cite{Tu20} for general $\kappa \in \R$.  We also emphasis that special cases of Theorem \ref{thm Neumann Eigenvalue} have been proved 
in \cite{AC13}\cite{AN12}\cite{CW94}\cite{CW95}\cite{Kroger92}\cite{LW19eigenvalue}\cite{LW19eigenvalue2}\cite{NV14}\cite{Valtorta12}\cite{ZY84}.
\end{remark}

\section{Modulus of Continuity Estimates for Dirichlet Boundary Condition}

To derive sharp estimates on the modulus of continuity of  solutions to \eqref{parabolic pde} with Dirichlet boundary condition, we fix one of two variables in the modulus of continuity estimate in Theorem \ref{Thm MC} to be inside the boundary and derive the following decay estimate. 
\begin{thm}\label{Thm Decay Intro}
Let $(M^n,g, e^{-f} d\mu_g)$ and $u$ be the same as in Theorem \ref{Thm MC}.  
Suppose that $M$ satisfies  $\Ric^N_f \geq (N-1)\kappa$ for $N\in [0,\infty)$  and $\kappa \in \R$, and $\p M$ satisfies $H_f \geq (N-1) \Lambda$ for $\Lambda \in \R$, or suppose that $\Ric_f \geq 0$ and $H_f \geq 0$ if $N = \infty$ (in this case we use the convention $T_{\kappa, \Lambda}=0$). 
Let $\vp:[0,R] \times \R_+ \to \R_+$ be a smooth function satisfying 
\begin{enumerate}
     \item $\vp_t \geq \a(\vp')\vp'' -(N-1) T_{\kappa, \Lambda} \b(\vp')\vp'$;
    \item $\vp'\geq 0$ on $[0,R] \times \R_+$, and $\vp(0, t)=0$ for $t\ge 0$.
\end{enumerate}
Define $$Z(x,t) :=u(x,t) -\vp\left(d(x,\p M), t\right).$$
If $Z(x,0)\leq 0$ on $M$, then $Z(x,t) \leq 0$ on $M\times [0,T)$.  
\end{thm}

\begin{proof}[Proof of Theorem \ref{Thm Decay Intro}]
By the same techniques as in the proof of Theorem \ref{Thm comparison distance to boundary} and
Theorem \ref{thm Neumann}, it's easy to see that under the assumptions of Theorem \ref{Thm Decay Intro}, the function $\vp\left(d(x,\p M), t \right)$ is a viscosity supersolution of \eqref{parabolic pde}. The desired estimate follows from the comparison principle for viscosity solutions since it holds true initially and on the boundary. 
\end{proof}

We present the proof of Theorem  \ref{Thm MC} now. The proof uses the comparison theorems for $d(x, \p M)$ and $d(x,y)$ proved in Sections 3 and 4. 

\begin{proof}[Proof of Theorem \ref{Thm MC}]
 For small $\eps>0$, consider the function $Z$ defined on $M \times M \times [0,T)$ by 
\begin{equation*}
    Z(x,y,t) =u(y,t) -u(x,t) -2\vp \left(\frac{d(x,y)}{2}, t \right) -\eps(1+t). 
\end{equation*}
Since $\vp_0$ is a modulus of continuity of $u(x,0)$, we have $Z(x,y,0) \leq -\eps$. If $Z$ ever becomes positive, there must be a first time $t_0 >0$ and points $x_0, y_0 \in M$ such that $Z(x_0,y_0,t_0) =0$ and $Z(x,y,t) \leq 0$ for all $x, y \in M$ and $t \leq t_0$. Clearly $x_0 \neq y_0$ as $Z(x,x,t) \leq -\eps$ for each $x\in M$ and $t\in [0,T)$. The Dirichlet boundary condition also rules out the possibility that both $x_0$ and $y_0$ lie on $\p M$. So we have three possibilities.

\textbf{Case 1:} Both $x_0$ and $y_0$ lie in the interior of $M$. 
In this case, the same argument as in the proof of \cite[Theorem 1]{AC13} with the comparison theorem there replaced by Theorem \ref{Thm comparison distance}, leads to a contradiction. Hence this case cannot occur. 

 \textbf{Case 2:} $x_0 \in \p M$ and $y_0$ lies in the interior of $M$. 
In this case, we have 
$$
u(y_0, t_0)-2\vp(\frac{d(x_0,y_0)}{2}, t_0)-\eps(1+t_0)=0.
$$
Using $\vp'\ge 0$ and $\vp''\le 0$, we estimate  that 
\begin{eqnarray}\label{con 1}
&&u(y_0, t_0)-\vp(d(y_0,\p M), t_0)\nonumber\\
&\ge& u(y_0, t_0)-2\vp(\frac{d(y_0,\p M)}{2}, t_0) \nonumber\\
&\ge& u(y_0, t_0)-2\vp(\frac{d(y_0,x_0)}{2}, t_0)\nonumber\\
&=&\eps (1+t_0).
\end{eqnarray}
Since $u(y,0) -\vp\left(d(y,\p M), 0 \right) \leq 0$, 
then by Theorem \ref{Thm Decay Intro}, we have 
$$
 u(y,t) -\vp\left(d(y,\p M), t \right)\le 0
$$
which contradicts with inequality \eqref{con 1} at $y=y_0$ and $t=t_0$. Therefore Case 2 cannot occur.

\textbf{Case 3:} $y_0 \in \p M$ and $x_0$ lies in the interior of $M$.
Similar argument as in Case 2 rules out this possibility. 
\end{proof}

\section{Two proofs of Theorem \ref{thm Dirichlet eigenvalue}}

We provide two proofs for Theorem \ref{thm Dirichlet eigenvalue} in this section. 
\subsection{First proof via decay estimates for parabolic equations}
Similarly as in \cite{AC13}, estimates on the modulus of continuity in Theorem \ref{Thm MC} lead to lower bound for the first Dirichlet eigenvalue. However, it does not give optimal lower bound for the first Dirichlet eigenvalue. 
Below we elaborate the difference between the first Dirichlet and the first closed or Neumnann eigenvalue for the Laplacian below. 
Recall that the idea of Andrew and Clutterbuck \cite{AC13} to detect the first nonzero eigenvalue (with either $\p M =\emptyset$ or Neumann boundary condition) of the Laplacian via the modulus of continuity estimates is by knowing how quickly the solutions to the heat equation decay. 
This is because we may solve 
\begin{equation*}
    \begin{cases}
    u_t =\Delta u, &\\
    u(x,0) =u_0(x), &
    \end{cases}
\end{equation*}
by expanding $u_0 =\sum_{i=0}^\infty a_i \vp_i$, where $\vp_i$ are eigenfunctions of the Laplacian (with Neumann boundary condition if $\p M \neq \emptyset$). Then the solution to the heat equation is given by $u(x,t) =\sum_{i=0}^\infty e^{-\l_i t} a_i \vp_i$. 
This does not converges to zero as $\l_0=0$, but the key idea is that $|u(x,t)-u(y,t)|$ does converges to zero and in fact   
$|u(x,t)-u(y,t)| \approx e^{-\l_1 t}$ as $t \to \infty$. Thus the main ingredient in \cite{AC13} is to establish the estimate 
\begin{equation*}
    |u(x,t)-u(y,t)| \approx C e^{-\bar{\l}_1 t}
\end{equation*}
for any solution to the heat equation, which is an easy consequence of the modulus of continuity estimates. Then taking $u(x,t)=e^{-\l_1 t} \vp_1(x)$ leads to 
\begin{equation*}
    |\vp_1(x) -\vp_1(y)| \leq C e^{(\l_1 -\bar{\l}_1)t},
\end{equation*}
which implies $\l \geq \bar{\l}_1$. 
If the Dirichlet boundary condition is posed, however, the solution $u(x,t)$ to the heat equation does converge to zero, as there is no constant term in the expansion of the initial data in terms of eigenfunctions. Thus to detect the first eigenvalue, it suffices to prove that any solution to the heat equation decays like $|u(x,t)| \leq C e^{-\bar{\l}_1 t} $. 
For this reason, sharp lower bound for the first Dirichlet are given in terms of the inradius $R$, rather than the diameter $D$, together with other curvature data.



When $Q[u]$ is homogeneous of degree $\gamma >0$, we get sharp decay estimates by comparing with self-similar solutions. 
\begin{prop}\label{Prop decay p-Laplacian}
Let $M$ and $u$ be the same as in Theorem \ref{Thm MC}. Assume $Q[u]$ is homogeneous of degree $\gamma>0$. Then we have the decay estimate
\begin{equation*}
u(x,t) \leq C e^{-t\l_1^{\frac{1}{\gamma-1}}}
\end{equation*}
where $C$ depends on $u(\cdot, 0)$, and $\l_1$ is the first eigenvalue of the one-dimensional problem \eqref{1D eq N finite}. 
\end{prop}


\begin{proof}[Proof of Proposition \ref{Prop decay p-Laplacian}]
Let $\vp$ be the eigenfunction associated to the eigenvalue $\l_1$. Since $\vp$ has positive derivative at $s=0$ and is positive for all $s\in (0, R]$, there exists $C>0$ depending only on $u(\cdot, 0)$ such that $u(x,0) \leq C \vp (d(x, \p M))$ for all $x\in M$. It's easy to verify the function $\psi(s,t)=C e^{-t\l_1^{\frac{1}{\gamma-1}}} \vp(s)$ satisfies all the requirements in Theorem \ref{Thm Decay Intro}, and we derive that 
\begin{equation*}
    u(x,t) \leq C e^{-t \l_1^{\frac{1}{\gamma-1}}} \vp(d(x,\p M)) \leq C \sup \vp \,e^{-t\l_1^{\frac{1}{\gamma-1}}}. 
\end{equation*}
\end{proof}
We can now give the first proof of Theorem \ref{thm Dirichlet eigenvalue}. 
\begin{proof}[First proof of Theorem \ref{thm Dirichlet eigenvalue}]
Let $\bar{\l}(Q)$ be the first Dirichlet eigenvalue of $Q[u]$ with eigenfunction $v(x)$, then $u(x,t)=e^{-t \bar{\l}(Q)^{\frac{1}{\gamma-1}}} v(x)$ satisfies \eqref{parabolic pde} on $M\times [0,\infty)$ with Dirichlet boundary condition.
By Proposition \ref{Prop decay p-Laplacian}, we have on $M\times [0,\infty)$, 
$$e^{-t \bar{\l}(Q)^{\frac{1}{\gamma-1}}} v(x) \leq C e^{-t\l_1^{\frac{1}{\gamma-1}}}.$$
Letting $t\to \infty$ implies $\bar{\l}(Q) \geq \l_1$. 
\end{proof}

\subsection{Second proof via comparison theorems for $d(x, \p M)$}
Our second proof used the new definition for the first Dirichlet eigenvalue given in \eqref{def principal eigenvalue} together with the comparison theorem for second derivatives of $d(x, \p M)$.  
\begin{proof}[Second proof of Theorem \ref{thm Dirichlet eigenvalue}]
We only deal with the case $N\in [n,\infty)$ here as the $N=\infty$ case is completely similar.  
Let $\l_1$ be the first eigenvalue of the one-dimensional problem \eqref{1D eq N finite} with $\vp$ be the corresponding eigenfunction. We must have $\vp'>0$ on $[0,R)$ since $\vp$ is the first eigenfunction. Then by Theorem \ref{Thm comparison distance to boundary}, the function $v(x)=\vp\left(d(x,\p M)\right)$ is a positive viscosity supersolution of $Q[u] = -\l |u|^{\gamma -1} u$. It follows from the definition of $\bar{\l}(Q)$ in \eqref{def principal eigenvalue} that we have $\bar{\l}(Q) \geq \l_1$. 
\end{proof}

\section{Gradient Estimates for Parabolic Equations}

In this section, we derive height-dependent gradient bounds for viscosity solutions of parabolic equations. Both the equations and the curvature conditions will be a bit more restrictive than in previous sections, but they are consistent with previous results in this direction obtained in \cite[Theorem 6]{Andrewssurvey15} for smooth solutions  and \cite[Theorem 4.1]{LW17} for viscosity solutions. 
We consider parabolic equations of the form
\begin{eqnarray}\label{paraeq}
\frac{\partial u}{\partial t}&=&\left[\alpha(|\n u|,u, t)\frac{\n_iu \n_ju}{|\n u|^2}
+ \beta(t)\left(\delta_{ij}-\frac{\n_iu \n_ju}{|\n u|^2}\right)\right]\n_i\n_ju\\ \nonumber
&&-\b(t)\langle \n f,\n u\rangle+q(|\n u|, u, t), 
\end{eqnarray}
where $\a$ and $\b$ are nonnegative functions.
It's easy to see that the \eqref{paraeq} covers the heat equation and the parabolic normalized $p$-Laplacian equation $u_t =\Delta^N_p u$. 

\begin{theorem}\label{thmh}
Let $(M^n,g, e^{-f} d\mu_g)$ be a compact smooth metric measure space  with diameter $D$ (possibly with smooth strictly convex boundary) and $\text{Ric}_f\ge \k$ for $\k \leq 0$.  
Suppose  $u: M\times [0,T)\rightarrow \mathbb{R}$ is a viscosity solution of \eqref{paraeq} (with Neumann boundary conditions if $\p M \neq \emptyset$). 
Let $\varphi: [0,D]\times [0,T)\rightarrow \mathbb{R}$ be a solution of
\begin{equation}\label{eqvpp}
\varphi_t=\alpha(\varphi', \vp, t)\varphi''-\k s \b(t) \vp'+q(\vp',\vp,t) 
\end{equation}
with $\vp'>0$, such that the
range of $u(\cdot, 0)$ is contained in $[\varphi(0,0), \varphi(D,0)]$. Let $\Psi(s,t)$ be given by
inverting $\varphi$ for each $t$, and assume that for all $x$ and $y$ in $M$,
$$\Psi(u(y,0),0)-\Psi(u(x,0),0)-d(x,y)\le 0.$$
Then
$$
\Psi(u(y,t),t)-\Psi(u(x,t),t)-d(x,y)\le 0.
$$
for all $x,y\in M$ and $t\in[0,T)$.
\end{theorem}

By letting $y$ approach $x$, we get
\begin{corollary}
Let $(M^n,g, e^{-f} d\mu_g)$, $u$ and $\vp$ be the same as in Theorem \ref{thmh}. Then 
\begin{equation*}
    |\n u(x, t)| \leq \vp'\left(\Psi(u(x,t),t) \right)
\end{equation*}
for all $(x,t) \in M \times [0, T)$.
\end{corollary}

We begin with a lemma about the behavior of parabolic semijets when composed with an increasing function. 
\begin{lemma}\label{Lemma}
Let $u$ be a continuous function. Let $\vp:\R \times [0, T) \to \R $ be a $C^{2,1}$ function with $\vp^{\prime} \geq 0$.
Let $\Psi:\R \times [0, T) \to \R $ be such that
$$\Psi(\vp(u(y,t),t),t)=u(y,t)$$
$$\vp (\Psi(u(y,t),t),t)=u(y,t)$$
(i) Suppose $(\tau, p, X) \in \mathcal{P}^{2,+}\Psi(u(y_0,t_0),t_0)$, then
$$(\vp_t+\vp^{\prime}\tau, \vp^{\prime}p, \vp^{\prime \prime}p \otimes p+ \vp^{\prime} X ) \in \mathcal{P}^{2,+}  u(y_0,t_0),$$
where all derivatives of $\vp$ are evaluated at $(\Psi(u(y_0,t_0)), t_0)$.

(ii) Suppose $(\tau, p, X) \in \mathcal{P}^{2,-}\Psi(u(y_0,t_0),t_0)$, then
$$(\vp_t+\vp^{\prime}\tau, \vp^{\prime}p, \vp^{\prime \prime}p \otimes p+ \vp^{\prime} X ) \in \mathcal{P}^{2,-}  u(y_0,t_0),$$
where all derivatives of $\vp$ are evaluated at $(\Psi(u(y_0,t_0)), t_0)$.

(iii) The same holds if one replaces the parabolic semijets by the their closures.
\end{lemma}
\begin{proof}
See \cite[Lemma 4.1]{LW17}. 
\end{proof}

\begin{proof}[Proof of Theorem \ref{thmh}]
The theorem is valid if we show that for any $\e>0$,
\begin{equation}
Z^{\e}(x,y,t):=\Psi(u(y,t),t)-\Psi(u(x,t),t)-d(x,y)-\frac{\e}{T-t}\le 0.\label{Ze}
\end{equation}
To prove inequality (\ref{Ze}), it suffices to show $Z^{\e}$ can not attain the maximum in $M\times M\times(0, T)$.
Assume by contradiction that there exist  $t_0 \in (0,T)$, $x_0$ and $y_0$ in $M$ at which the function
$Z^{\e}$
attains its maximum. Notice that the Neumann condition, convexity of $\p M$, and the positivity of $\vp'$ guarantees that $x_0$ and $y_0$ are in $M$ if $\p M \neq \emptyset$.

Take $\rho(x, y)$ defined as in Definition \ref{def-rho} with $\eta=1$ .   Then the function
$$\Psi(u(y,t),t)-\Psi(u(x,t),t)-\rho(x,y)-\frac{\eps}{T-t}$$ has a local maximum at $(x_0,y_0,t_0)$.
If $\e >0$, then we necessarily have $x_0 \neq y_0$.
By the parabolic maximum principle Theorem \ref{max prin} for semicontinuous functions on manifolds,
for any $\lambda >0$, there exist $X, Y$ satisfying
 \begin{eqnarray*} 
  (b_1, \n_y \rho(x_0, y_0), X) &\in& \overline{\mathcal{P}}^{2,+} \Psi(u(y_0,t_0), t_0),\\
  (-b_2, -\n_x \rho(x_0, y_0), Y) &\in& \overline{\mathcal{P}}^{2,-} \Psi(u(x_0,t_0), t_0),
 \end{eqnarray*}
 \begin{equation}\label{Dt}
 b_1+b_2=\frac{\e}{(T-t_0)^2},
 \end{equation}
 and
\begin{equation}\label{Hessian inequality for quasilinear}
  -\left(\lambda^{-1}+\left\|S\right\| \right)I \leq
    \begin{pmatrix}
   X & 0 \\
   0 & -Y
   \end{pmatrix}
   \leq S+\lambda S^2,
  \end{equation}
 where $S=\n^2 \rho(x_0,y_0)$.
By Lemma \ref{Lemma}, we have
\begin{equation*}
   (b_1 \vp^{\prime}(z_{y_0}, t_0)+\vp_t(z_{y_0}, t_0), \vp^{\prime}(z_{y_0}, t_0) e_n(s_0), \vp^{\prime}(z_{y_0}, t_0)X+\vp''(z_{y_0}, t_0)e_n(s_0)\otimes e_n(s_0))
  \end{equation*}
  and
 \begin{equation*}
 (-b_2 \vp^{\prime}(z_{x_0}, t_0) +\vp_t(z_{x_0}, t_0), -\vp^{\prime} (z_{x_0}, t_0) e_n(0), \vp^{\prime}(z_{x_0}, t_0) Y+\vp''(z_{x_0}, t_0)e_n(0)\otimes e_n(0)) 
 \end{equation*}
 are in $\overline{\mathcal{P}}^{2,+} u(y_0,t_0)$
 and $\overline{\mathcal{P}}^{2,-} u(x_0,t_0)$ respectively.
Where
$z_{x_0}=\Psi(u(x_0, t_0), t_0)$, $z_{y_0}=\Psi(u(y_0, t_0), t_0)$, and we used the first variation  formula $\n \rho(x_0,y_0)=(-e_n(0), e_n(s_0))$.
Since $u$ is both a subsolution and a supersolution of (\ref{paraeq}), we have
\begin{eqnarray}\label{pb1}
&&b_1 \vp^{\prime}(z_{y_0}, t_0)+\vp_t(z_{y_0}, t_0)\nonumber\\
&\le& \operatorname{tr}\Big(\vp^{\prime}(z_{y_0}, t_0) A_1 X+
\vp''(z_{y_0}, t_0)A_1e_n(s_0)\otimes e_n(s_0)\Big)\nonumber\\&&
-\b(t_0)\vp'(z_{y_0},t_0)\langle \n f(y_0), e_n(s_0)\rangle+q(\vp'(z_{y_0},t_0),\vp(z_{y_0},t_0), t_0),
\end{eqnarray}
and
\begin{eqnarray}\label{pb2}
&&-b_2 \vp^{\prime}(z_{x_0}, t_0) +\vp_t (z_{x_0}, t_0)\nonumber\\
&\ge& \operatorname{tr}\Big( \vp^{\prime}(z_{x_0}, t_0)A_2Y+\vp''(z_{x_0}, t_0)A_2 e_n(0)\otimes e_n(0)\Big)\nonumber\\
&&-\b(t_0)\vp'(z_{x_0},t_0)\langle \n f(x_0), e_n(0)\rangle+q(\vp'(z_{x_0},t_0),\vp(z_{x_0},t_0), t_0),
\end{eqnarray}
where $$
A_1=\left(
    \begin{array}{cccc}
      \beta( t_0) & \cdots & 0 & 0 \\
     \vdots & \vdots & \vdots & \vdots \\
     0 & \cdots &\beta(t_0) & 0 \\
      0 & \cdots & 0 & \alpha(|\vp^{\prime}(z_{y_0}, t_0)|,\vp(z_{y_0}, t_0), t_0) \\
    \end{array}
  \right),
$$
and
$$
A_2=\left(
    \begin{array}{cccc}
      \beta( t_0) & \cdots & 0 & 0 \\
     \vdots & \vdots & \vdots & \vdots \\
     0 & \cdots &\beta(t_0) & 0 \\
      0 & \cdots & 0 & \alpha(|\vp^{\prime}(z_{x_0}, t_0)|,\vp(z_{x_0}, t_0), t_0) \\
    \end{array}
  \right).
$$
Set
$$C=\left(
    \begin{array}{cccc}
      \beta(t_0) & \cdots & 0 & 0 \\
     \vdots & \vdots & \vdots & \vdots \\
     0 & \cdots &\beta(t_0) & 0 \\
      0 & \cdots & 0 & 0 \\
    \end{array}
  \right),$$
and simple calculation shows $\left(
                           \begin{array}{cc}
                             A_1 & C \\
                            C & A_2 \\
                           \end{array}
                         \right)\ge 0$.
Then we conclude from (\ref{Dt}), (\ref{pb1}) and (\ref{pb2}) that
\begin{eqnarray*}
&&\frac{\e}{(T-t_0)^2}=b_1+b_2\\
&\le& \frac{\operatorname{tr}(\vp^{\prime}(z_{y_0}, t_0)A_1X)+\operatorname{tr}(A_1 e_n(s_0)\otimes e_n(s_0))\vp''(z_{y_0}, t_0)-
\vp_t(z_{y_0}, t_0)}{\vp^{\prime}(z_{y_0}, t_0)}\\
&&+\frac{\operatorname{tr}(-\vp^{\prime}(z_{x_0}, t_0)A_2 Y)-\operatorname{tr}(A_2 e_n(0)\otimes e_n(0))\vp''(z_{x_0}, t_0)+
\vp_t(z_{x_0}, t_0)}{\vp^{\prime}(z_{x_0}, t_0)}\\
&&+\frac{q(\vp'(z_{y_0},t_0),\vp(z_{y_0},t_0), t_0)}{\vp^{\prime}(z_{y_0}, t_0)}-\frac{q(\vp'(z_{x_0},t_0),\vp(z_{x_0},t_0), t_0)}{\vp^{\prime}(z_{x_0}, t_0)}\\
&&-\b(t_0)\Big(\langle \n f(y_0), e_n(s_0)\rangle-\langle \n f(x_0), e_n(0)\rangle\Big)\\
&=& -\b(t_0)\Big(\langle \n f(y_0), e_n(s_0)\rangle-\langle \n f(x_0), e_n(0)\rangle\Big)+\operatorname{tr}\left[\left(
                                                        \begin{array}{cc}
                                                          A_1& C \\
                                                         C &  A_2 \\
                                                        \end{array}
                                                      \right)
\left(
\begin{array}{cc}
 X & 0 \\
  0 & -Y \\
   \end{array}
   \right)
\right]\\
&&+\frac{\vp_t(z_{x_0}, t_0)-\alpha(\vp'(z_{x_0},t_0),\vp(z_{x_0}, t_0), t_0)\vp''(z_{x_0},t_0)-q(\vp'(z_{x_0},t_0),\vp(z_{x_0},t_0),t_0)}{\vp^{\prime}(z_{x_0}, t_0)}\\
&&-\frac{\vp_t(z_{y_0}, t_0)-\alpha(\vp'(z_{y_0},t_0),\vp(z_{y_0}, t_0),t_0)\vp''(z_{y_0},t_0)-q(\vp'(z_{y_0},t_0),\vp(z_{y_0},t_0),t_0)}{\vp^{\prime}(z_{y_0}, t_0)}\\
&\le&-\b(t_0)\Big(\langle \n f(y_0), e_n(s_0)\rangle-\langle \n f(x_0), e_n(0)\rangle\Big)+\b(t_0)\k (z_{y_0}-z_{x_0})\\
&&+\operatorname{tr}\left[\left(
                                                        \begin{array}{cc}
                                                       A_1& C \\
                                                         C &  A_2 \\
                                                        \end{array}
                                                      \right)S
\right]
+\lambda\operatorname{tr}\left[\left(
                                                        \begin{array}{cc}
                                                      A_1 & C \\
                                                         C &  A_2 \\
                                                        \end{array}
                                                      \right)S^2
\right]
\end{eqnarray*}
where we have used the inequality (\ref{Hessian inequality for quasilinear}) and the equation (\ref{eqvpp}) of $\vp$.

Direct calculation gives
\begin{eqnarray*}
\operatorname{tr}\left[\left(
                                                        \begin{array}{cc}
                                                       A_1& C \\
                                                         C &  A_2 \\
                                                        \end{array}
                                                      \right)S
\right]&=&\b(t_0)\sum_{i=1}^{n-1}\n^2 \rho\Big((e_i(0),e_i(s_0)),(e_i(0),e_i(s_0))\Big)\\
&&+\a(\vp'(z_{y_0},t_0),\vp(z_{y_0},t_0),t_0)\n^2 \rho\Big((0,e_n(s_0)),(0,e_n(s_0))\Big)\\
&&+\a(\vp'(z_{x_0},t_0),\vp(z_{x_0},t_0),t_0)\n^2 \rho\Big((e_n(0),0),(e_n(0),0)\Big)
\end{eqnarray*}
Since
$$
\n^2 \rho\Big((e_n(0),0),(e_n(0),0)\Big)=0,\quad \n^2 \rho\Big((0,e_n(s_0)),(0,e_n(s_0))\Big)=0
$$
and
\begin{eqnarray*}
&&\sum_{i=1}^{n-1}\n^2 \rho\Big((e_i(0),e_i(s_0)),(e_i(0),e_i(s_0))\Big)\\
&=&\int_0^{s_0}(n-1)(\eta')^2-\eta^2 \operatorname {Ric}(e_n,e_n) \, ds\\
&\le&-\k s_0+\Big(\langle \n f(y_0), e_n(s_0)\rangle-\langle \n f(x_0), e_n(0)\rangle\Big)
\end{eqnarray*}
where we used (\ref{2nd2}) and (\ref{eta2}).
Therefore we conclude 
\begin{eqnarray*}
\frac{\e}{(T-t_0)^2}&\le& \k\b(t_0)(z_{y_0}-z_{x_0}-s_0)+\lambda\operatorname{tr}\left[\left(
                                                        \begin{array}{cc}
                                                      A_1 & C \\
                                                         C &  A_2 \\
                                                        \end{array}
                                                      \right)S^2
\right]\\
&\le&\lambda\operatorname{tr}\left[\left(
                                                        \begin{array}{cc}
                                                      A_1 & C \\
                                                         C &  A_2 \\
                                                        \end{array}
                                                      \right)S^2
\right],
\end{eqnarray*}
where we have used $\k\le 0$ and $$z_{y_0}-z_{x_0}-s_0>0$$ by the assumption.  
Then we get the  contradiction by letting $\lambda\rightarrow 0$ . Therefore (\ref{Ze}) is true, hence completing the proof.
\end{proof}


\section{Gradient Estimates for Elliptic Equations}

We derive height-dependent gradient estimate for  elliptic quasi-linear equations. For elliptic equations, we can deal with the slightly more general quasi-linear operator
\begin{equation}\label{eq1.1}
\L_f (u, \n u, \n^2 u)=0,
\end{equation}
where the operator $\L_f$ is defined by 
\begin{eqnarray*}
\L_f (u, \n u, \n^2 u) &=& \left[\a(u,|\n u|)\frac{\n_i u \n_j u}{|\n u|^2}+\b(u,|\n u|)\left(\delta_{ij}-\frac{\n_i u \n_j u}{|\n u|^2}  \right) \right] \n_i\n_j u \\
&& -\b(u,|\n u|)\langle \n u, \n f\rangle +b(u,|\n u|),
\end{eqnarray*}
where $\alpha$ and $\beta$ are nonnegative functions, $\beta(s,t)>0$ for $t>0$.
\begin{thm}\label{Thm1.3}
Let $(M^n, g,f)$ be a closed Bakry-Emery manifold with $\Ric+\nabla^2 f \geq \kappa g$ for some $\kappa \leq 0$. Let $u$ be a viscosity solution of the equation \eqref{eq1.1}. Let $\varphi: [a,b]\rightarrow [\inf u, \sup u]$ be a $C^2$ solution of 
\begin{itemize}
     \item[(i)] $\alpha(\vp, \varphi')\varphi'' -\kappa\,  t\,  \b(\vp,\vp')+b(\vp, \vp') =0$ on $[a,b]$;
    \item[(ii)] $\vp(a) =\inf u$, \text{  } $\vp(b)=\sup u$, \text{  } $\vp' >0$ on $[a,b]$. 
\end{itemize}
Let $\Psi$ be the inverse of $\vp$. Then we have 
\begin{equation*}
    \Psi(u(y)) -\Psi(u(x)) -d(x,y) \leq 0,
\end{equation*}
for all $x,y \in M$.
\end{thm}

As an immediate corollary, by letting $y$ approach $x$, we get the following gradient estimate:
\begin{corollary}
Under the assumptions of Theorem \ref{Thm1.3}, we have 
\begin{equation*}
    |\n u(x)| \leq \vp'(\Psi(u(x)))
\end{equation*}
for all $x\in M$. 
\end{corollary}

\subsection{The case $\kappa \leq 0$}
\begin{proof}[Proof of Theorem \ref{Thm1.3}]
We argue by contradiction and suppose that 
$$m := \max_{M\times M} \left\{ \Psi(u(y)) -\Psi(u(x)) -d(x,y) \right\} >0.$$
The positive maximum must be attained at some point $(x_0,y_0) \in M \times M$ with $x_0 \neq y_0$, since the function $\Psi(u(y)) -\Psi(u(x)) -d(x,y)$ is continuous and vanishes on the diagonal of $M\times M$. 
We replace $d(x,y)$ by $\rho(x,y)$ to apply maximum principle. From the definition of $\rho(x,y)$, we see that $d(x,y) \leq \rho(x,y)$ in $U(x_0,y_0)$ with equality at $(x_0,y_0)$. 
Thus we have 
\begin{equation*}
    \Psi(u(y)) -\Psi(u(x)) -\rho(x,y) \leq m
\end{equation*}
on $U(x_0, y_0)$ and with equality at $(x_0,y_0)$.
Now we can apply the maximum principle for semicontinuous functions on manifolds to conclude that for any $\lambda >0$, there exist $X\in Sym^2(T^*_{x_0}M)$ and $Y\in Sym^2(T^*_{y_0}M)$ such that 
\begin{align*}
    (\n_y\rho(x_0,y_0), Y ) & \in  \overline{J}^{2,+}(\Psi(u(y_0)), \\
    (-\n_x \rho(x_0,y_0), X) &  \in  \overline{J}^{2,-}(\Psi(u(x_0)),
\end{align*}
and 
 \begin{equation}\label{eq2.1}
    \begin{pmatrix}
   X & 0 \\
   0 & -Y
   \end{pmatrix}
   \leq S+\lambda S^2,
  \end{equation}
where $S=\n^2 \rho(x_0,y_0)$.
The first variation formula of arc length implies
\begin{equation*}
    \n_y \rho(x_0,y_0) =e_n(s_0) \text{ and }  \n_x \rho(x_0,y_0) =-e_n(0). 
\end{equation*}
By Lemma 8 in \cite{AX19}, we get 
\begin{align*}
    \left(\vp'(z_{y_0})e_n(s_0), \vp'(z_{y_0}) Y +\vp''(z_{y_0}) e_n(s_0) \otimes e_n(s_0) \right) & \in  \overline{J}^{2,+}u(y_0), \\
    \left(\vp'(z_{x_0})e_n(0), \vp'(z_{x_0}) X +\vp''(z_{x_0}) e_n(0) \otimes e_n(0) \right) & \in  \overline{J}^{2,-}u(x_0),
\end{align*}
where $z_{y_0} = \Psi (u(y_0))$ and $z_{x_0} = \Psi (u(x_0))$. 

The fact that $u$ is a viscosity solution of \eqref{eq1.1} implies  
\begin{eqnarray*}
  &&\tr(\vp'(z_{y_0})A_2 Y +\vp''(z_{y_0})A_2e_n(s_0) \otimes e_n(s_0) ) +b(\vp(z_{y_0}), \vp'(z_{y_0}))\\
  &\ge&\beta(\vp(z_{y_0}), \vp'(z_{y_0}))\vp'(z_{y_0})\langle e_n(s_0), \n f(y_0) \rangle   
\end{eqnarray*}
 and 
\begin{eqnarray*}
   && \tr(\vp'(z_{x_0})A_1 X +\vp''(z_{x_0})A_1 e_n(0) \otimes e_n(0) ) +b(\vp(z_{x_0}), \vp'(z_{x_0})) \\
   &\leq& -\beta(\vp(z_{x_0}), \vp'(z_{x_0}))\vp'(z_{x_0})\langle e_n(0), \n f(x_0) \rangle ,
\end{eqnarray*}
where 
\begin{equation*}
A_1=\left(
    \begin{array}{cccc}
      \b(\vp(z_{x_0}), \vp'(z_{x_0})) & \cdots & 0 & 0 \\
     \vdots & \vdots & \vdots & \vdots \\
     0 & \cdots &\b(\vp(z_{x_0}), \vp'(z_{x_0})) & 0 \\
      0 & \cdots & 0 & \a(\vp(z_{x_0}), \vp'(z_{x_0})) \\
    \end{array}
  \right),
\end{equation*}
and 
\begin{equation*}
A_2=\left(
    \begin{array}{cccc}
      \b(\vp(z_{y_0}), \vp'(z_{y_0})) & \cdots & 0 & 0 \\
     \vdots & \vdots & \vdots & \vdots \\
     0 & \cdots &\b(\vp(z_{y_0}), \vp'(z_{y_0})) & 0 \\
      0 & \cdots & 0 & \a(\vp(z_{y_0}), \vp'(z_{y_0})) \\
    \end{array}
  \right).
\end{equation*}

Therefore, 
\begin{eqnarray*}
  && \a(\vp(z_{y_0}), \vp'(z_{y_0})) \vp''(z_{y_0}) +b(\vp(z_{y_0}), \vp'(z_{y_0})) + \vp'(z_{y_0}) \tr \left(\begin{pmatrix}
   0 & C \\
   C & A_2
   \end{pmatrix} \begin{pmatrix}
   X & 0 \\
   0 & -Y
   \end{pmatrix} \right) \\
    &\ge& \beta(\vp(z_{y_0}), \vp'(z_{y_0}))\vp'(z_{y_0})\langle e_n(s_0), \n f(y_0) \rangle, 
\end{eqnarray*} 
where  
\begin{equation*}
C=\left(
    \begin{array}{cccc}
      \b(\vp(z_{y_0}), \vp'(z_{y_0})) & \cdots & 0 & 0 \\
     \vdots & \vdots & \vdots & \vdots \\
     0 & \cdots &\b(\vp(z_{y_0}), \vp'(z_{y_0})) & 0 \\
      0 & \cdots & 0 & 0 \\
    \end{array}
  \right).
\end{equation*}
Similarly, 
\begin{eqnarray*}
   && \a(\vp(z_{x_0}), \vp'(z_{x_0})) \vp''(z_{x_0}) +b(\vp(z_{x_0}), \vp'(z_{x_0}))  + \vp'(z_{x_0}) \tr \left(\begin{pmatrix}
   A_1 & 0 \\
   0 & 0
   \end{pmatrix} \begin{pmatrix}
   X & 0 \\
   0 & -Y
   \end{pmatrix} \right) \\
    &\le & \beta(\vp(z_{x_0}), \vp'(z_{x_0}))\vp'(z_{x_0})\langle e_n(0), \n f(x_0) \rangle.
\end{eqnarray*}
Combing the above two inequalities, 
\begin{eqnarray*}
0 &\leq& \left.\frac{\a(\vp,\vp') \vp'' +b(\vp, \vp')}{\b(\vp,\vp') \vp'} \right|_{z_{y_0}} - \left.\frac{\a(\vp,\vp') \vp'' +b(\vp,\vp')}{\b(\vp,\vp') \vp'} \right|_{z_{x_0}} \\
&& -\langle e_n(s_0), \n f(y_0)\rangle + \langle e_n(0), \n f(x_0) \rangle \\
&& + \frac{1}{\b(\vp(z_{y_0}), \vp'(z_{y_0}))} \tr \left(\begin{pmatrix}
   0 & C \\
   C & A_2
   \end{pmatrix} \begin{pmatrix}
   X & 0 \\
   0 & -Y
   \end{pmatrix} \right) \\
&& +\frac{1}{\b(\vp(z_{x_0}), \vp'(z_{x_0}))}\tr \left(\begin{pmatrix}
   A_1 & 0 \\
   0 & 0
   \end{pmatrix} \begin{pmatrix}
   X & 0 \\
   0 & -Y
   \end{pmatrix} \right).
\end{eqnarray*} 
Letting 
\begin{eqnarray*}
W &=& \frac{1}{\b(\vp(z_{y_0}),\vp'(z_{y_0}))} \begin{pmatrix}
   0 & C \\
   C & A_2
   \end{pmatrix}    +\frac{1}{\b(\vp(z_{x_0}),\vp'(z_{x_0}))} \begin{pmatrix}
   A_1 & 0 \\
   0 & 0
   \end{pmatrix}  \\
&=& \begin{pmatrix}
   I_{n-1} & 0 & I_{n-1} & 0 \\
   0 & \left.\frac{\a(\vp,\vp')}{\b(\vp,\vp')}\right|_{z_{x_0}} & 0 & 0 \\
   I_{n-1} & 0 & I_{n-1} & 0 \\
   0 & 0 & 0 &  \left.\frac{\a(\vp,\vp')}{\b(\vp,\vp')}\right|_{z_{y_0}}
   \end{pmatrix}
\end{eqnarray*}
It's easy to see that $W$ is a positive semi-definite matrix. 
Using \eqref{eq2.1}, we obtain that
\begin{eqnarray*}
0 &\leq& \left.\frac{\a(\vp,\vp') \vp'' +b(\vp, \vp')}{\b(\vp,\vp') \vp'} \right|_{z_{y_0}} 
- \left.\frac{\a(\vp,\vp') \vp'' +b(\vp,\vp')}{\b(\vp,\vp') \vp'} \right|_{z_{x_0}} \\
&& -\langle e_n(s_0), \n f(y_0)\rangle + \langle e_n(0), \n f(x_0) \rangle \\
&& + \tr(WS) +\lambda \tr(WS^2) 
\end{eqnarray*}
and
\begin{eqnarray*}
\tr(WS) &=& \sum_{i=1}^{n-1} \n^2 \rho \left( (e_i(0),e_i(s_0)), (e_i(0),e_i(s_0)) \right) \\
&& + \left.\frac{\a(\vp,\vp') }{\b(\vp,\vp') } \right|_{z_{x_0}} \n^2 \rho \left( (e_n(0), 0), (e_n(0),0) \right)   \\
&& + \left.\frac{\a(\vp,\vp') }{\b(\vp,\vp') } \right|_{z_{y_0}} \n^2 \rho \left( (0, e_n(s_0)), (0, e_n(s_0)) \right) \\
&=& - \int_0^{s_0} \Ric(e_n,e_n) ds 
\end{eqnarray*}
where we used the variation formulas 
\begin{eqnarray*}
\n^2 \rho \left( (e_n(0), 0), (e_n(0),0) \right) =0, \quad
\n^2 \rho \left( (0, e_n(s_0)), (0, e_n(s_0)) \right) =0.
\end{eqnarray*}
 and (\ref{2nd2}) with $\eta(s)=1$.

Finally, we get 
\begin{eqnarray*}
0 &\leq& \left.\frac{\a(\vp,\vp') \vp'' +b(\vp, \vp')}{\b(\vp,\vp') \vp'} \right|^{z_{y_0}}_{z_(x_0)} +\lambda \tr(WS^2)  \\
&& -\langle e_n(s_0), \n f(y_0)\rangle + \langle e_n(0), \n f(x_0) \rangle - \int_0^{s_0} \Ric(e_n,e_n) ds   \\
&\leq & \kappa z_{y_0} -\kappa z_{x_0} +\lambda \tr(WS^2)  -\kappa s_0 \\
&=& \kappa \left (\Psi(u(y_0)) -\Psi(u(x_0)) -d(x_0,y_0) \right) +\lambda \tr(WS^2) \\
&=& \kappa \, m +\lambda \tr(WS^2), 
\end{eqnarray*}
where we used the curvature condition $\Ric+\n^2 f\ge \k$ in the second inequality.
Since $\kappa <0$ and $m>0$, we get a contradiction by letting $\lambda \to 0$. 
\end{proof}

\subsection{The case $\kappa >0$}

The argument given for $\kappa \leq 0$ in the previous section does not lead to a contradiction when $\kappa >0$. As in \cite{AX19}, we may not be able to show that any solutiong $\vp$ to the one-dimensional equation is a barrier in the case $\kappa >0$. However, we prove that for some family of solutions to the one-dimensional equation, the property of being barriers can be extended smoothly in the family. Moreover, this phenomenon holds for any $\kappa \in \R$, regardless of its sign. 

\begin{thm}\label{Thm3.1}
Let $(M^n, g,f)$ be a closed Bakry-Emery manifold with $\Ric+\nabla^2 f \geq \kappa g$ for some $\kappa > 0$. Let $u$ be a $C^3$ solution of equation \eqref{eq1.1}.
Assume $\a, \b$ are $C^2$ functions. 
Suppose $\vp_c:[a_c,b_c] \to [\inf u, \sup u]$ is a family of $C^2$ solutions of the one-dimensional equation
\begin{equation}
    \a(\vp,\vp') \vp'' -\kappa \, t \, \b(\vp, \vp') \vp' +b(\vp, \vp') =0
\end{equation}
on $[a,b]$ which satisfies 
\begin{itemize}
    \item[(i)] $\vp_c(a_c)=\inf u, \vp_c(b_c)=\sup u, \vp_c'>0 \text{ on } [a_c,b_c];$
    \item[(ii)] $\vp'_c$ is uniformly large for $c \gg c_u$;
    \item[(iii)] $\vp_c$ depends smoothly on $c \in (c_u, \infty)$.
\end{itemize}
Let $\Psi_c$ be the inverse of $\vp_c$. Then we have 
\begin{equation}\label{eq3.2}
    \Psi_c(u(y)) -\Psi_c(u(x)) -d(x,y) \leq 0,
\end{equation}
for all $x,y \in M$ and $c\in (c_u, \infty)$.
\end{thm}
\begin{corollary}
Under the assumptions of Theorem \ref{Thm3.1}, we have 
\begin{equation*}
    |\n u(x)| \leq \vp'_c \left(\Psi_c(u(x)) \right), 
\end{equation*}
for all $x\in M$ and $c\in (c_u, \infty)$. 
\end{corollary}

We prove the following lemma, which will be needed in the proof of Theorem \ref{Thm3.1}. 
\begin{lemma}\label{lemma3.1}
Let $u$ be a $C^3$ solution of \eqref{eq1.1}. Let $x\neq y$ with $d(x,y) < \text{inj}(M)$, the injectivity radius of $M$. Let $\gamma_0: [0,s_0] \to M$ be the length-minimizing geodesic from $x$ to $y$, and choose Fermi coordinate as before. Let $z_y=\Psi(u(y))$ and $z_x =\Psi(u(x))$. 
Then 
\begin{equation*}
    Z(x,y): =\Psi(u(y)) -\Psi(u(x)) -d(x,y)
\end{equation*}
 satisfy 
\begin{eqnarray*}
\F[Z]&:=&  \left.\frac{\a(\vp,\vp')}{\b(\vp,\vp')}\right|_{z_y} \n^2_{((0,e_n),(0,e_n))} Z + \left.\frac{\a(\vp,\vp')}{\b(\vp,\vp')}\right|_{z_x} \n^2_{((e_n,0),(e_n,0))} Z  \\
&& + \sum_{i-1}^{n-1} \n^2_{((e_i,e_i),(e_i,e_i))} Z \\
&=& \left.\frac{\a(\vp,\vp')\vp'' +b(\vp,\vp')}{\b(\vp,\vp')\vp'}\right|^{z_x}_{z_y} 
+ \n Z * \n Z +P * \n Z,
\end{eqnarray*}
where the coefficients of $\n Z * \n Z$ and $P *\n Z$ are $C^1$ functions. 
\end{lemma}
\begin{proof}
This is a special case of Lemma 15 in \cite{AX19}.
\end{proof}

\begin{proof}[Proof of Theorem \ref{Thm3.1}]
We argue by contradiction and assume that \eqref{eq3.2} does not hold for some $c_0 >c_u$. 

Let $\Delta = \{(x,x) : x\in M\}$ be the diagonal of $M\times M$ and consider a manifold $\hat{M}$ with boundary, which is a natural compactification of $(M\times M)\setminus \Delta$. 
As a set, $\hat{M}$ is the disjoint union of the $(M\times M)\setminus \Delta$ and the unit sphere bundle $SM:=\{(x,v) : x \in M, v \in T_xM \}$. The manifold with structure is defined by the atlas generated by all charts for $(M\times M)\setminus \Delta$, together with the charts $\hat{Y}$ from $SM\times (0,r)$ defined by taking a chart $Y$ for $SM$ and setting $\hat{Y}(z,s):= \left(\exp(sY(z)), \exp(-s Y(z)) \right)$.

For simplicity of notations, we write $\vp =\vp_c$ and $\Psi =\Psi_c$ in the rest of the proof. 
Define the function $\hat{Z}$ on $\hat{M}$ by
\begin{equation*}
  \hat{Z}(x,y) =\frac{Z(x,y)}{d(x,y)} \text{ for } (x,y) \in (M\times M )\setminus \Delta,
\end{equation*}
and 
\begin{equation*}
  \hat{Z}(x, v) =\frac{\n_v u(x)}{\vp'(\Psi(u(x)))} -1 \text{ for } (x,v) \in SM. 
\end{equation*}
It's easy to see that function $\hat{Z}$ is continuous on $\hat{M}$. 
Assumption (ii) implies that $\hat{Z} \leq 0$ on $\hat{M}$ for all $c$ sufficiently large. 
So let $c_1$ be the smallest number such that 
$\hat{Z} \leq 0$ on $\hat{M}$ for all $c \geq c_1$, i.e.,
\begin{equation*}
    c_1 = \inf \{ t > c_u : \hat{Z} \leq 0 \text{ on } \hat{M} \text{ for all } c \in (t, \infty) \}. 
\end{equation*}
By continuity, we have $c_1 > c_0$, which we shall prove lead to a contradiction.  
For $c=c_1$, there will be two cases. 

\textbf{Case 1:} $\hat{Z}(x_0,y_0) =0 $ for some $x_0 \neq y_0$. 

By Lemma \ref{lemma3.1}, we have at $(x_0,y_0)$, 
\begin{eqnarray*}
\F[Z] +\n Z *\n Z +P*\n Z  \geq \kappa (z_{y_0}-z_{x_0}) = \kappa \, d(x_0,y_0) >0.
\end{eqnarray*}
This contradicts the fact that $Z$ attains its maximum at $(x_0,y_0)$, thus ruling out Case 1.

\textbf{Case 2:} $Z(x,y) < 0$ for all $x\neq y \in M$ and $\hat{Z}(x_0, v_0) =0$ for some $(x_0,v_0) \in SM$. 

In this case, by Lemma \ref{lemma3.1}, we have for $x\neq y$ close enough to each other, 
\begin{eqnarray*}
\F[Z] +\n Z *\n Z +P*\n Z  \geq \kappa (z_{y_0}-z_{x_0}) =  \kappa  \, d(x_0,y_0) + \kappa Z \geq \kappa Z. 
\end{eqnarray*}
The Hopf maximum principle in \cite{Hill70} applies to this situation and yields at $(x_0,v_0)$, 
\begin{eqnarray*}
0 > \n_{(0,v)} Z(x,x) 
&=&\lim_{t\to 0} \frac{Z(x,\exp(tv)) -Z(x,x)}{t} \\
&=& \lim_{t\to 0} \frac{\Psi(u(\exp(tv))) -\Psi(u(x)) - d(x, \exp_x(tv))}{t} \\
&=& \frac{\n_vu(x)}{\vp'(z_x)} -1.
\end{eqnarray*}
This is a contradiction to $\hat{Z}(x_0,v_0) =\frac{\n_{v_0}u(x_0)}{\vp'(z_{x_0})} -1 =0$. 
Therefore, Case 2 is impossible. 
\end{proof}



\bibliographystyle{plain}
\bibliography{ref}

\end{document}